\documentclass[a4paper,11pt]{amsart}

\usepackage{amsmath,amsthm,amscd,amssymb}
\usepackage{stmaryrd}

\usepackage{tikz}
\usepackage[all]{xy}

\newcommand{\abs}[1]{\left\lvert #1 \right\rvert} 

\usepackage[backref]{hyperref}
\hypersetup{
  colorlinks   = true,          
  urlcolor     = blue,          
  linkcolor    = purple,          
  citecolor   = blue             
}

\newtheorem{innercustomthm}{Theorem}
\newenvironment{customthm}[1]
  {\renewcommand\theinnercustomthm{#1}\innercustomthm}
  {\endinnercustomthm}

\usepackage{fullpage,enumerate}
\usepackage[euler-digits]{eulervm}

\theoremstyle{definition}
\newtheorem{defn}{Definition}[section]
\newtheorem{ex}[defn]{Example}
\newtheorem{notation}[defn]{Notation}
\newtheorem{remark}[defn]{Remark}

\theoremstyle{plain}
\newtheorem{theorem}[defn]{Theorem}
\newtheorem{corollary}[defn]{Corollary}
\newtheorem{lemma}[defn]{Lemma}
\newtheorem{prop}[defn]{Proposition}

\newcommand{\Z}{\mathbb{Z}}
\newcommand{\Q}{\mathbb{Q}}

\newcommand{\R}{\mathbb{R}}

\renewcommand{\P}{\mathbb{P}}

\newcommand{\relint}{\textnormal{relint}\,}

\newcommand{\Star}{\textnormal{Star}}
\newcommand{\pr}{\textnormal{pr}}

\newcommand{\dist}{\textnormal{dist}}

\newcommand{\mk}[1]{M_{0,#1}^\textnormal{trop}}

\newcommand{\curly}[1]{\mathcal{#1}}

\DeclareMathOperator{\trop}{trop}

\newcommand{\wo}{\setminus} 

\newcommand{\nocontentsline}[3]{}
\newcommand{\tocless}[2]{\bgroup\let\addcontentsline=\nocontentsline#1{#2}\egroup}

\begin{document}
\title[Tropical moduli spaces of rational weighted stable curves]{Moduli spaces of rational weighted stable curves and tropical geometry}
\author {Renzo Cavalieri, Simon Hampe, Hannah Markwig, Dhruv Ranganathan}
\address {Renzo Cavalieri, Department of Mathematics, Colorado State University}
\email{renzo@math.colostate.edu}
\address{Simon Hampe, Institut f\"ur Mathematik, Technische Universit\"at Berlin}
\email{hampe@math.tu-berlin.de}
\address {Hannah Markwig, Fachbereich Mathematik, Eberhard Karls Universit\"at T\"ubingen}
\email {hannah@mathematik.uni-tuebingen.de}
\address {Dhruv Ranganathan, Department of Mathematics, Yale University}
\email{dhruv.ranganathan@yale.edu}

\begin{abstract}
We study moduli spaces of rational weighted stable tropical curves, and their connections with Hassett spaces. Given a vector $w$ of weights, the moduli space of tropical $w$-stable curves can be given the structure of a balanced fan if and only if $w$ has only heavy and light entries. In this case, the tropical moduli space can be expressed as the Bergman fan of an explicit graphic matroid. The tropical moduli space can be realized as a geometric tropicalization, and as a Berkovich skeleton, its algebraic counterpart. This builds on previous work of Tevelev, Gibney and Maclagan, and Abramovich, Caporaso, and Payne. Finally, we construct the moduli spaces of heavy/light weighted tropical curves as fibre products of unweighted spaces, and explore parallels with the algebraic world. 
\end{abstract}
\maketitle

\setcounter{tocdepth}{2}
\tableofcontents
\pagebreak
\setcounter{section}{0}
\numberwithin{equation}{section}

\section{Introduction}

\subsection{Main results}\label{subsec-wsc}
Let $w=(w_1,\ldots,w_n)\in\Q^n$ be a vector of weights satisfying $0< w_i\leq 1$ for all $i$ and $\sum w_i>2$. Let $C$ be a tree of $\mathbb{P}^1$'s, with $n$ smooth marked points $p_1,\ldots,p_n$. Marked points are not necessarily distinct: a subset of marks is allowed to coincide precisely when the sum of the corresponding weights is no larger than $1$. The curve $C$ is \textit{$w$-stable}, if for each component $T$ of $C$, the sum 
\[
\sum_{i\; ; \;p_i\in T}w_i+\#\textnormal{nodes}>2.
\]
These spaces were introduced by Hassett~\cite{Has03} in the context of the log minimal model program. In particular, he proves that there exists a smooth projective scheme representing the moduli problem of $w$-weighted stable curves. This scheme is denoted $\overline{M}_{0,w}$. 

If $w=(1^n)$, then $\overline{M}_{0,w}=\overline{M}_{0,n}$ is the well-known moduli space of stable curves. In this case, there is an elegant connection to the space of leaf-labelled metric trees, obtained by \textit{geometric tropicalization}. 

In this paper, we study tropical analogues of moduli spaces of rational weighted stable curves and their relation to the algebraic moduli spaces.
We introduce {\it tropical rational weighted stable curves} in the natural way, by defining the combinatorial type of a $w$-stable tropical curve to be the dual graph of a $w$-stable curve, keeping track of the weights on the marked ends (Definition~\ref{tws}). Parameter spaces for tropical rational weighted stable curves carry the structure of abstract cone complexes, with graph contractions giving rise to natural gluing between cones. 

We address the following questions:
\begin{enumerate}[(A)]
 \item For which values of $w$ can the cone complex $M^{\trop}_{0,w}$ be given the structure of a balanced fan embedded into a vector space?
\item When $M^{\trop}_{0,w}$ is a balanced fan, can it be realized as a tropicalization of the classical moduli spaces of $w$-weighted stable curves?
\item In the above cases, can the toric variety associated to the fan $M^{\trop}_{0,w}$ be used to define the $w$-stable compactification of the locus of nonsingular marked curves $M_{0,w}$?
\end{enumerate}
\noindent
These questions are addressed completely. We begin by observing that, for any fixed $n$,  the notion of stability is governed by a finite set of first degree inequalities in the weights. The parameter space for the weights is subdivided into polyhedral chambers. For any $w,w'$ in the same chamber, $M^{\trop}_{0,w}$ is canonically isomorphic to $M^{\trop}_{0,w'}$.

A weight vector $w$ has only \textit{heavy and light} weights if it is in the same chamber as $(1^f, \epsilon^t)$, for $\epsilon$ such that $t\cdot \epsilon<1$ (see Definition \ref{lhw}). The following answers Question (A). 

\begin{customthm}{I}
 The cone complex $M^{\trop}_{0,w}$ can be given the structure of a balanced fan in a vector space if and only if $w$ has only heavy and light entries.
\end{customthm}

\noindent
Our next result generalizes work of Ardila and Klivans~\cite{AK06} to the weighted case.

\begin{customthm}{II}
The moduli space $M^{\trop}_{0,w}$ with heavy/light weights is the Bergman fan of a graphic matroid.
\end{customthm}
\noindent
The above two results are combined and stated more precisely in the main body of the paper as Theorem~\ref{thm-heavyandlight}.

Finally, we answer Questions (B) and (C). See Theorem~\ref{thm-tropicalizing} in the text.
 
\begin{customthm}{III}
Let $w$ be a heavy/light weight vector. There exists a toric variety $X(\Delta)$ with torus $T_w$, and an embedding $M_{0,w}\hookrightarrow T_w$, such that
\begin{enumerate}[(T1)]
\item The tropicalization of $M_{0,w}$ with respect to the given embedding is $M_{0,w}^{\trop}$.
\item The fan $\Delta$ is naturally identified with $M_{0,w}^{\trop}$.
\item The closure of $M_{0,w}$ in $X(\Delta)$ is Hassett's compactification $\overline M_{0,w}$.
\end{enumerate}
\end{customthm}

In Section~\ref{sec: skeleta} we explore the connection with Berkovich skeletons, in the spirit of the results obtained in~\cite{ACP12,U}. This establishes a compatibility between the geometric tropicalization approach employed in this text, and the perspective of skeletons of analytic spaces.

A natural variation of weighted stable curves, also considered by Hassett, arises by allowing the weight vector $w$ to have zero entries. In this case, the moduli spaces are described by fibre products of the universal curve of $M_{0,w^+}$, where $w^+$ is the subcollection of positive entries of $w$. We discuss a tropical analogue of this situation in Theorem \ref{thm-tropfibre} and Corollary \ref{cor-fibreheavylight}. 

The main results rely on a careful study of the classical and combinatorial \textit{reduction morphisms} from the spaces $\overline M_{0,n}$ and $M_{0,n}^{\trop}$ to the weighted spaces. 
\subsection{Context and motivation}
\label{cnm}
Tropical geometry has become a successful tool in algebraic geometry, with exciting applications in the study of enumerative geometry, moduli spaces, and linear series on algebraic curves. Tropical enumerative geometry began with Mikhalkin's celebrated Correspondence Theorem relating numbers of plane curves satisfying point conditions with their tropical counterparts~\cite{Mi03}. The results sparked substantial interest in tropical moduli spaces.

The space $M^{\trop}_{0,n}$ is a polyhedral complex parametrizing leaf-labelled metric trees. It was first studied in connection with phylogenetics~\cite{BHV01}, and later, in relation with the geometry of the tropical Grassmanian~\cite{SS04a}. The cone complex $M^{\trop}_{0,n}$ can be embedded into a vector space, and given the structure of a balanced fan, by assigning weight $1$ to each top dimensional cone. In \cite{GKM07,Mi07}, $M_{0,n}^{\trop}$ is discussed in analogy with the classical moduli space, with the goal of understanding its tropical intersection theory. However, the connection between the algebraic and tropical moduli spaces runs much deeper. Building on previous work of Tevelev~\cite{Tev07}, Gibney and Maclagan~\cite{GM07} exhibit $M_{0,n}^{\trop}$ as a \textit{tropicalization}. They find an embedding of $M_{0,n}$ into a torus, such that the tropicalization of $ M_{0,n}$ is a balanced fan $\Sigma$, such that $\Sigma\cong M^{\trop}_{0,n}$. The closure of $M_{0,n}$ in the toric variety $X(\Sigma)$ is  $\overline M_{0,n}$.

This phenomenon falls under the theory of tropical compactification, as developed by Hacking, Keel, and Tevelev~\cite{Tev07,HKT09}. The philosophy is that the features of a suitable (e.g. toroidal) compactification are inherent in the tropicalization of a subvariety of a torus, i.e.\ ``the tropicalization knows a good compactification''. The $M_{0,n}$ case is particularly nice as the tropicalization has an intrinsic modular interpretation. With this connection between the spaces $\overline{M}_{0,n}$ and $\mk{n}$ established, the orbit--cone correspondence of the ambient toric variety induces an order reversing bijection between dimension $k$ strata in the classical moduli space, and codimension $k$ cones in the tropical moduli space. Furthermore, using techniques from toric intersection theory, Katz~\cite[Section 7]{Kat09} shows that the intersection theory on $\overline M_{0,n}$ can be related to the toric intersection theory on the ambient toric variety, and hence to the tropical intersection theory on $M^{\trop}_{0,n}$. With the spaces of rational weighted stable curves, we exhibit a new class of moduli spaces, having a natural tropical modular skeleton. It seems reasonable to expect that intersection numbers on the heavy/light spaces $\overline M_{0,w}$ can be computed tropically, following work of Katz.

\begin{remark}
The spaces $\overline{\mathcal M}_{g,w}$ for higher genus are defined analogously. In general however, such spaces do not admit nice embeddings to toric varieties. Nonetheless, the moduli space of tropical $w$-stable curves can be viewed as a skeleton of the analytification of $\overline{\mathcal M}_{g,w}$. Since this paper first appeared, Ulirsch has exhibited this connection~\cite{U}. In genus $0$, we observe that the compactification of $\overline M_{0,w}$ can be obtained combinatorially. It is intriguing to ask whether there exists a similar relationship between ${\mathcal M}_{g,w}$ and $\overline{\mathcal M}_{g,w}$.
\end{remark}

\subsection{Outline}
Section~\ref{sec: tropical-moduli} is devoted to projections of $\mk{n}$ and their relation to Bergman fans with the nested set subdivision of a building set defined in terms of graphs. We start with a subsection reviewing the necessary preliminaries, subdivided in a part about matroids, Bergman fans and nested sets and a part about basics in tropical geometry and the moduli space of abstract $n$-marked tropical curves. Subsection~\ref{section_moduli} contains original work. The main results are Theorem~\ref{moduli_cor_mainresult}, stating that a certain projection of $\mk{n}$ is the Bergman fan of a graphic matroid in a nested set subdivision and Theorem \ref{thm-heavyandlight} that states that this projection is an embedding of $\mk{w}$ as a balanced fan if $w$ contains only heavy and light points. If $w$ does not only contain heavy and light points, then we cannot embed $\mk{w}$ as a balanced fan. This is stated in Theorem~\ref{thm-heavyandlight}.

In Section~\ref{sec: trop-wsc} we explore the tropicalization of the algebraic moduli spaces $\overline{M}_{0,w}$. We start with a subsection reviewing the preliminaries in the  context of the tropicalization of $\overline{M}_{0,n}$. Subsection~\ref{subsec: geom-trop} contains our main result of this section, Theorem~\ref{thm-tropicalizing}, which states that in the case of heavy and light points, we can embed $\overline{M}_{0,w}$ into a toric variety defined by $\mk{w}$. This space exists as a balanced fan as a consequence of the results in Section~\ref{sec: tropical-moduli}. In this embedding, the tropicalization of $\overline M_{0,w}$ intersected with the torus is shown to be canonically isomorphic to the tropical moduli space. We also discuss the situation where we do not only have heavy and light points. In Section~\ref{sec: LosevManin}, we consider a special case: the  Losev--Manin spaces,  with exactly two heavy and only light points otherwise. They are toric varieties, which also follows from Theorem \ref{thm-tropicalizing}, since the corresponding tropical spaces are just subdivisions of $\R^{n-3}$. We then consider the relationship with Berkovich skeletons

In Section~\ref{section_fibre}, we construct the moduli spaces of heavy/light weighted tropical curves as fibre products of spaces of unweighted tropical curves. The motivating idea here is to replace the light points with weight $0$ points. The resulting classical spaces, studied by Hassett, are singular. The heavy/light spaces can be obtained as natural desingularizations of the weight $0$ spaces. The weight $0$ spaces, in turn, may be expressed as fibre products of spaces of unweighted curves. The main results of this section rely on several general structure results for fibre products of Bergman fans associated to graphic matroids.

\noindent
\textbf{Acknowledgements.} The authors are grateful to Martin Ulirsch for discussing his related work. The first author is grateful for the support from  NSF grant DMS-1101549, NSF RTG grant 1159964. The second and third author were partially supported by DFG-grant 4797/1-2, the third author in addition by DFG-grant 4797/6-1. The fourth author acknowledges several helpful conversations with Dustin Cartwright, Dave Jensen, Sam Payne, and Tif Shen. The paper was substantially improved by the comments of two anonymous referees, to whom we are very grateful. 

\section{Tropical moduli spaces of rational weighted curves}\label{sec: tropical-moduli}

Let $C$ be a rational abstract tropical curve. That is, $C$ is a leaf-labelled metric tree. Edges adjacent to leaves are called \textit{ends}, and are metrized as $[0,\infty]$. The other edges are called \textit{bounded edges} and have finite length. Let $w=(w_1,\ldots,w_n)\in\Q^n$ be a vector of weights satisfying $0< w_i\leq 1$ for all $i$.  

\begin{defn}\label{tws}
Let $V$ be a vertex of $C$ and assume that there are $l$ bounded edges adjacent to $V$ as well as the ends with markings $i\in I$ for a subset $I\subset [n]$. We say that $C$ is \textit{$w$-stable} at $V$ if $\sum_{i\in I}w_i+l>2$. We say that $C$ is \textit{$w$-stable} if it is $w$-stable at every vertex. We define $M^{\trop}_{0,w}$ to be the set of all tropical rational $w$-stable curves.
\end{defn}

The set $M^{\trop}_{0,w}$ can be given the structure of an abstract cone complex by gluing the cones corresponding to a combinatorial type (i.e.\ a tree without metrization data) in the way dictated by the underlying tropical curves.
As such, for any $w$, $M^{\trop}_{0,w}$ is a sub-cone complex of $M^{\trop}_{0,n}$. This definition generalizes the definition of abstract $n$-marked rational tropical curve: if we set $w=(1^n)$, then the stability condition is that every vertex is at least $3$-valent. 

\begin{defn}\label{lhw}
 Let $w\in \Q^n$.   Let $i \in [n]$. 
 \begin{itemize}
  \item[(H)] We call $i$ \textit{heavy in $w$}, if for all $j \neq i$ we have $w_i + w_j > 1$.
  \item[(S)] We call $i$ \textit{small in $w$}, if $w_i + w_j >1$ implies that $j$ is heavy in $w$.
 \end{itemize}
If, in addition, the total weight of the small points is less than $1$, we say that they are \textit{light}.
\end{defn}

We will often consider weight vectors that are \textit{heavy/light} (resp. \textit{heavy/small}), meaning that  each entry of $w$ is either heavy or light (resp. heavy or small). 


\subsection{Tropical moduli spaces of rational curves as Bergman fans}

\subsubsection{Matroids, Bergman fans, and nested sets}
Matroids abstract the concept of linear independence of subsets of a set of vectors. Important examples are matroids of point configurations defined by the usual linear independence and matroids of graphs, where dependence is defined in terms of cycles. For a detailed introduction to matroids, see for instance~\cite{Kat14}.

To any matroid $M$ on  a ground set $E(M)$ we associate a polyhedral fan, the \textit{Bergman fan} $B(M) \subseteq \R^{|E(M)|}$ in the following manner:
$$B(M) := \{w \in \R^{\abs{E(M)}}; M_w \textnormal{ is loop-free}\},$$
where $M_w$ is the matroid on $E(M)$ whose bases are all bases $B$ of $M$ of minimal $w$-weight $\sum_{i \in B} w_i$. Ardila and Klivans~\cite{AK06} showed that $B(M)$ is a polyhedral cone complex that coincides with the order complex of the lattice of flats of $M$. More precisely, for each chain of flats in $M$
$$\curly{F} = \emptyset \subsetneq F_1 \subsetneq \dots \subsetneq F_r = E,$$
we let $C_{\curly{F}}$ be the cone in $\R^{\abs{E}}$ spanned by rays $v_{F_1},\dots,v_{F_{r-1}}$, with lineality space spanned by $v_{F_r}$. Here $v_F = - \sum_{i \in F} e_i$, where $e_i$ is the $i$-th standard basis vector. The collection of these cones forms a fan whose support is $B(M)$. We call this particular polyhedral structure the \textit{chains-of-flats} subdivision of $B(M$).

A useful equivalent definition is the following~\cite[Proposition~2.5]{FS05}.
\[
B(M) := \{w;\; \max\{w_i; i \in C\} \textnormal{ is attained at least twice for all circuits }C\}.
\]
Feichtner and Sturmfels demonstrate multiple polyhedral structures that can be placed on this fan using the theory of~\textit{building sets}.

\begin{defn}
 Let $\curly{F}$ be the lattice of flats of a matroid $M$. For two flats $F,F' \in \curly{F}$ we write $[F,F'] := \{G \in \curly{F}: F \subseteq G \subseteq F'\}$.
 A \textit{building set} for $\curly{F}$ is a subset $\curly{G}$ of $\curly{F} \wo \{\emptyset\}$ such that the following holds:
 
 \noindent
 For any $F \in \curly{F} \wo \{\emptyset\}$, let $\{G_1,\dots,G_k\}$ be the maximal elements of $\curly{G}$ contained in $F$. Then there is an isomorphism of partially ordered sets:
 $$\varphi_F: \prod_{j=1}^k [\emptyset,G_j] \to [\emptyset,F],$$
 where the $j$-th component of $\varphi_F$ is the inclusion $[\emptyset,G_j] \subseteq [\emptyset,F]$.
 
 A subset $\curly{S}$ of a building set $\curly{G}$ is called \textit{nested}, if for any set of incomparable elements $F_1,\dots,F_l$ in $\curly{S}$ with $l \geq 2$, the join $F_1 \vee \dots \vee F_l$ is not an element of $\curly{G}$.
\end{defn}

\begin{remark}
 The nested sets of a building set $\curly{G}$ form an abstract simplicial complex (a subset of a nested set is a nested set). We can assign to each flat $F$ the vector $v_F \in \R^{|E|}$ defined above, and accordingly a cone for each nested set of $\curly{G}$. It has been shown in \cite[Theorem 4.1]{FS05} that this produces a polyhedral fan whose support is $B(M)$.
\end{remark}

\begin{notation}
 Note that each Bergman fan contains the linear space $L$, spanned by the vector $(1,\dots,1)$. It is standard to quotient by this lineality space, and study the resulting space
 $$B'(M) := B(M) / L.$$
 \end{notation}
 
 \begin{remark}
  It is a well-known fact that the set of Bergman fans is closed under cartesian products. In fact, if $M,M'$ are matroids on ground sets $E,E'$, then
$$B(M) \times B(M') = B(M \oplus M'),$$
where $M \oplus M'$ is the matroid on $E \amalg E'$, whose bases are disjoint unions of bases of $M$ and $M'$.
 \end{remark}

\subsubsection{Facts about tropical geometry and $M^{\trop}_{0,n}$}\label{prelim-trop}

\begin{defn}
 A \textit{tropical fan} $(X,\omega)$ is a rational pure $d$-dimensional polyhedral fan in $\R^n$, with a \textit{weight function} $\omega: X^{(d)} := \{\sigma \in X; \dim \sigma = d\} \to \Z_{>0}$, fulfilling the balancing condition in the sense of~\cite[Section~3.4]{MS09}.
 
 We consider two balanced fans $(X,\omega),(X',\omega')$ to be \textit{equivalent}, if they have a common refinement: if there exists a balanced fan $(X'',\omega'')$, whose cones are contained in cones of $X$ and $X'$ respectively; and $\omega''$ is compatible with $\omega$ and $\omega'$ in the natural way.
\end{defn}

\begin{remark}
 In this paper, we consider balanced fans that arise as Bergman fans of matroids. Given a matroid $M$, $B(M)$, equipped with the chain-of-flats subdivision is a fan. It is balanced with weight function identically $1$.
 
It is also known that $(B(M), \omega \equiv 1)$ is \textit{irreducible}, i.e.\ any balanced fan of the same dimension, which is contained in $B(M)$, must be equal to $B(M)$ as a set and its weight function must be an integer multiple of $w$. See, for instance, \cite[Lemma~2.4]{FR10}.
\end{remark}

\begin{defn}
 A \textit{tropical morphism} between balanced fans $X \subset \R^n,Y \subset \R^m$ is a map of fans, respecting the weight function, and the integral structure on $X$ and $Y$. More explicitly, it is a map $f: X \to Y$ mapping cones to cones, induced by a linear map $\Z^n \to \Z^m$. We say that $f$ is an \textit{isomorphism}, if it is bijective and respects the weights of $X$ and $Y$.
\end{defn}

It is often necessary to understand the local structure of a fan, near a given cone.

\begin{defn}
 Let $(X,\omega)$ be a tropical fan and $\tau$ a cone of $X$. We define the local picture of $X$ around $\tau$ to be the weighted fan 
 $$\Star_X(\tau) := (\{\Pi(\sigma); \tau \leq \sigma; \sigma \textnormal{ a cone of } X\},\omega_\Star),$$where $\Pi: \R^n \to \R^n / V_\tau$ is the residue map and the weight function is defined by $\omega_\Star: \Pi(\sigma) \mapsto \omega(\sigma)$. It is easy to see that $\Star_X(\tau)$ is a tropical fan (with respect to the lattice $\Z^n/ \Lambda_\tau$).
 
 If $p$ is a point in the support of $X$, we  also consider
 $$\Star_X(p) := (\{\sigma - p; p \in \sigma; \sigma \textnormal{ a cone of } X\},\omega_\Star),$$
 where $\omega_\Star$ is defined in the same way and $\sigma - p = \{ a - p; a \in \sigma\}$.
\end{defn}

We now briefly discuss some properties of the tropical moduli space $\mk{n}$ and its embedding as a balanced fan. For a more detailed account, see~\  \cite{GKM07, Mi07, SS04a}.
\begin{remark}
Let $C$ be an abstract $n$-marked tropical curve and let $v_i$ and $v_j$ be vertices incident to the leaves labeled $i$ and $j$. Let $\dist(i,j)$ denote the distance between these vertices. It has been shown that the vector
  $$d(C) = (\dist(i,j))_{i <j} \in \R^{\binom{n}{2}} / \Phi(\R^n)$$
  identifies $C$ uniquely, where $\Phi: \R^n \to \R^{\binom{n}{2}}, x \mapsto (x_i + x_j)_{i < j}$.

 Let $C$ be an abstract $n$-marked tropical curve with a single edge, inducing a partition or \textit{split} on the leaves $[n] = I \amalg I^c$. We denote the corresponding ray spanned by $d(C)$ by $v_I = v_{I^c}$. A $d$-dimensional cone of $\mk{n}$ corresponds to a combinatorial type of curve with $d$ bounded edges. If these edges introduce splits $I_1,\dots,I_d$, the cone is spanned by rays $v_{I_1},\dots,v_{I_d}$. The fan $\mk{n}$ can be balanced with all weights equal to $1$.

The fan structure described above, with one cone for each combinatorial type is the coarsest fan structure that can be defined on $\mk{n}$. We call it the \textit{combinatorial subdivision}.
 
It has been shown in~\cite{AK06,FR10} that 
 $$\mk{n} \cong B'(K_{n-1}),$$
where $K_{n-1}$ is the complete graph on $n-1$ vertices. For later use, we now describe the explicit isomorphism between $\mk{n}$ and $B'(K_{n-1})$. That is, we  describe the image of a vector $v_F$, when $F$ is any flat. See Figure \ref{fig_flat_to_curve} for an example.
 
For convenience, throughout the rest of this paper, we label the vertices of $K_{n-1}$ by $2,\ldots, n$. 

A flat $F$ of $K_{n-1}$ is a union of complete subgraphs on disjoint vertex sets $V_1,\dots,V_t$. Denote by $C_F$ the tropical $n$-marked curve constructed in the following manner: Attach $t$ bounded edges $e_1,\dots,e_t$ of length 1 to a common vertex $v$. To the end of $e_i$, attach leaves $\{j; j \in V_i\}$. Then attach all leaves $\{j; j \in [n] \wo \bigcup_{i=1}^t V_i\}$ to the vertex $v$. Then
 $$v_F \mapsto C_F.$$
 
Conversely, if we pick a ray $v_I$ of $\mk{n}$ with $1 \notin I$, it corresponds to the flat $F_I$, which is the complete graph on vertices in $I$. We will see in Example \ref{ex_mn_nested_subdiv} how the combinatorial subdivision of $\mk{n}$ can be expressed as a nested set subdivision of $B'(K_{n-1})$.

 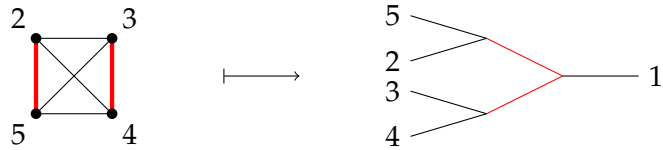
\begin{figure}[ht]
 \centering
 \begin{tikzpicture}
  \matrix[column sep = 10mm]{
    \draw (0,0) -- (1,0) -- (1,-1) -- (0,-1) -- (0,0);
    \draw (0,0) -- (1,-1);
    \draw (1,0) -- (0,-1);
    \draw[ultra thick,red] (0,0) -- (0,-1);
    \draw[ultra thick,red] (1,0) -- (1,-1);
    \fill[black] (0,0) circle (2pt) node[above left] {2};
    \fill[black] (1,0) circle (2pt) node[above right] {3};
    \fill[black] (1,-1) circle (2pt) node[below right] {4};
    \fill[black] (0,-1) circle (2pt) node[below left] {5};&
    \draw[|->] (0,-0.5) -- (1,-0.5); &
    \draw[red] (0,-0.5) -- (-1,0);
    \draw[red] (0,-0.5) -- (-1,-1);
    \draw (0,-0.5) -- (1,-0.5) node[right]{1};
    \draw (-1,0) -- (-2,-0.3) node[left]{2};
    \draw (-1,0) -- (-2,0.3) node[left]{5};
    \draw (-1,-1) -- (-2,-1.3) node[left]{4};
    \draw (-1,-1) -- (-2,-0.7) node[left]{3};

    \\
  };
 \end{tikzpicture}  
 \caption{Assigning a curve to a flat of $K_4$: The thick red lines indicate the flat, which consists of two complete subgraphs. We get one bounded edge for each subgraph. They are connected at a common vertex to which we attach all remaining leaves.}\label{fig_flat_to_curve}
 \end{figure}

\end{remark}

\subsection{Tropical moduli spaces of weighted curves as Bergman fans}\label{section_moduli}

Let $w$ be a weight vector. To obtain the cone complex $\mk{w}$, we wish to contract unstable rays and their adjacent cones. It is a natural thought to use a projection of $\mk{n}$ embedded as a fan contracting these rays; we denote this projection by $\pr_w$. It turns out however that $\pr_w$ may contract too many cones: this is the case if and only if $w$ does not only have heavy and light entries.
The projection $\pr_w(\mk{n})$ is still an interesting balanced fan which we can understand in terms of the Bergman fan of a graphic matroid. Moreover, as we  see in  Section \ref{sec: skeleta}, the image of this projection can be realized as the Berkovich skeleton of the classical moduli space $\overline M_{0,w}$ with respect to an appropriate toroidal structure. 

Our first goal is to gain a clearer understanding of the role of the complete graph $K_{n-1}$ in the $M_{0,n}^{\trop}$ case, and how it generalizes to the weighted case. 

An $n$-marked trivalent rational tropical curve is $w$-stable if all vertices with exactly two leaves $i$ and $j$ attached fulfill the condition $w_i + w_j > 1$. Consequently, two weight vectors $w,w'$ produce the same set of stable combinatorial types of trivalent curves, if $w_i + w_j > 1 \iff w_i' + w_j' > 1$ for all pairs $i,j$. It is therefore reasonable to expect that a graph encoding these conditions will be meaningful towards understanding spaces of weighted stable tropical curves.

\begin{defn}
 Let $w $ be a weight vector. We define the \textit{total weight graph} $G_t(w)$ to be the graph on vertices $\{1,\dots,n\}$ where two vertices $i,j$ are connected by an edge, if and only if $w_i + w_j > 1$. 
 \end{defn}

The notions of heavy and small can be efficiently expressed in terms of the total weight graph. Let $i \in [n]$. 
 \begin{itemize}
  \item  $i$ is heavy if $i$ is connected to all other vertices in $G_t(w)$. 
  \item $i$ is small if $i$ is only connected to heavy vertices.
 \end{itemize}
 
In parallel with the $M_{0,n}^{\trop}$ case, we seek to have a graph on $(n-1)$ vertices. 

\begin{defn} 
The \textit{reduced weight graph}, denoted $G(w)$, is the graph obtained from $G_t(w)$ by deleting any single heavy vertex.
\end{defn}

\begin{remark}
Dropping different heavy vertices we obtain isomorphic graphs, since all heavy vertices are incident to all other vertices. We will see in Corollary \ref{moduli_cor_needstable} that if there is no heavy weight, $\pr_w(\mk{n})$ does not have the ``expected dimension'', which is $n-3$. We will henceforth assume that $w_1 = 1$, and the reduced weight graph is constructed by deleting the vertex $1$ from $G_t(w)$.
\end{remark}

\begin{remark}\label{rem-prw}
There is a natural projection morphism induced by the fact that $G(w)$ is a subgraph of the complete graph $K_{n-1}$:
 \[
\widetilde \pr_w: \mk{n}\cong B'(K_{n-1}) \to B'(G(w)),
 \]
which forgets the coordinates corresponding to edges not lying in $G(w)$. We can see that $\widetilde \pr_w$ is precisely the projection morphism $\pr_w$ discussed above.

A ray $v_I$ with $1 \notin I$ which is contracted because it corresponds to an unstable curve comes from a flat $F_I$ given by the induced subgraph on the vertices in $I$. The fact that $v_I$ is unstable means that $\sum_{i \in I}w_i\leq 1$ which implies $w_i+w_j\leq 1$ for all $i,j\in I$, so no edge connecting two vertices in $I$ belongs to $G(w)$.
Vice versa, if no edge connecting two vertices in $I$ belongs to $G(w)$, then $w_i+w_j\leq 1$ for all $i,j\in I$ and hence $v_{\{i,j\}}$ corresponds to an unstable curve and is contracted by $\pr_w$ for all $i,j\in I$. But then $v_I=\sum_{\{i,j\}\subset I} v_{\{i,j\}}$ is also contracted, see~\cite[Lemma~2.6]{KM07}.
\end{remark}

We first relate the projection $\pr_w(\mk{w})$ to the nested set subdivision induced by the \textit{building set of 1-connected flats} on the graph $G(w)$.

\begin{defn}
 Let $G$ be a (connected) graph and $M_G$ the corresponding matroid. We define a \textit{building set of 1-connected flats}:
 $$\curly{G}_G := \{F \in \curly{F}(M_G); G_{\mid F} \textnormal{ is a connected graph}\},$$
 where $G_{\mid F}$ is the restriction of $G$ to the edges contained in $F$.
 \end{defn}

The above definition depends not only on the matroid, but on the presentation of this matroid as a graphic matroid. Recall that two non-isomorphic graphs $G$, $G'$ may yield the same matroid. In this case $\curly{G}_G$, $\curly{G}_{G'}$ might be different building sets. We also warn the reader that the matroidal notion for \textit{connected sets} differs from the above definition. A connnected set of a graphic matroid $M(G)$ is a set whose underlying graph is 2-connected.
 
 To see that $\curly{G}_G$ is a building set, let $F$ be a flat of $M_G$. The maximal elements of $\curly{G}_G$ contained in $F$ are exactly the connected components $G_1,\dots,G_k$ of the subgraph $G_{\mid F}$. Any flat $F' \subseteq F$ can also be partitioned into its connected components, which in turn must be subsets of the connected components of $F$. We see that there is an isomorphism
 $$\prod_{j=1}^k [\emptyset,G_j] \cong [\emptyset,F].$$

We are now ready to state the first main result of this section. 

\begin{theorem}\label{moduli_cor_mainresult}
Let $w $ be a weight vector and assume $w$ has at least two heavy entries. Then $\pr_w(\mk{n})=B'(G(w))$.
 Furthermore, the combinatorial types of curves in $\pr_w(\mk{n})$ correspond to the cones of $B'(G(w))$ in the nested set subdivision with respect to $\curly{G}_{G(w)}$, the building set of 1-connected flats.
\end{theorem}

Before proceeding to the proof, we consider some examples and prove certain auxiliary results.

\begin{ex}\label{ex_mn_nested_subdiv}
 Let $w = (1,\dots,1)$. Since no ray of $\mk{n}$ becomes unstable, $\pr_w$ is  the identity map.
In this case, $G_t(w) = K_n$, the complete graph on $n$ vertices, and the reduced graph is $G(w) = K_{n-1}$. We already observed that the Bergman fan corresponding to this graph is $B'(K_{n-1})\cong \mk{n}=\pr_w(\mk{n})$. Theorem \ref{moduli_cor_mainresult} tells us that the combinatorial subdivision of $\mk{n}$ corresponds to the nested set subdivision of $B'(K_{n-1})$ with respect to $\curly{G}_{K_{n-1}}$. This is seen as follows. It is well known that the combinatorial subdivision is the coarsest possible polyhedral structure on $\mk{n}$. Feichtner and Sturmfels~\cite[Theorem 5.3]{FS05} showed that it is obtained as the nested set subdivision with respect to the building set of connected flats. In the case of graphic matroids, this means choosing all flats whose underlying graph is 2-connected. However in this particular case, flats are disjoint unions of complete graphs on at least $3$ vertices, so they are 1-connected if and only if they are 2-connected, excluding the case of the complete graph on two vertices. 
\end{ex}

The following example demonstrates that $\pr_w(\mk{n})=B'(G(w))$ may not be the embedding of the cone complex $\mk{w}$ as a fan.
\begin{ex}\label{moduli_ex_badtypes} Let $w = (1,1,3/4, 3/4,1/4)$. The reduced weight graph is a $K_3$ with an additional edge attached to it (connecting the remaining 1 and the vertex with weight 1/4). The corresponding Bergman fan is $B'(G(w))\cong \mk{4} \times \R$.  

The 1-connected flats of $G(w)$ are depicted in Figure \ref{fig_connected_flats_example}. A nested set is formed either by a chain or two incomparable flats whose join is not connected, i.e. by two vertex-disjoint flats. Only $F_3, F_4$ are vertex-disjoint, so all other nested sets are formed by chains. Hence we obtain the following 8 cones:
\begin{align*}
 \sigma_1 &:= \textnormal{cone}(v_{F_3},v_{F_4}) & \sigma_5 &:= \textnormal{cone}(v_{F_2},v_{F_7}) \\
 \sigma_2 &:= \textnormal{cone}(v_{F_1},v_{F_5}) & \sigma_6 &:= \textnormal{cone}(v_{F_3},v_{F_5}) \\
 \sigma_3 &:= \textnormal{cone}(v_{F_1},v_{F_6}) & \sigma_7 &:= \textnormal{cone}(v_{F_4},v_{F_6}) \\
 \sigma_4 &:= \textnormal{cone}(v_{F_2},v_{F_5}) & \sigma_8 &:= \textnormal{cone}(v_{F_4},v_{F_7})
\end{align*}

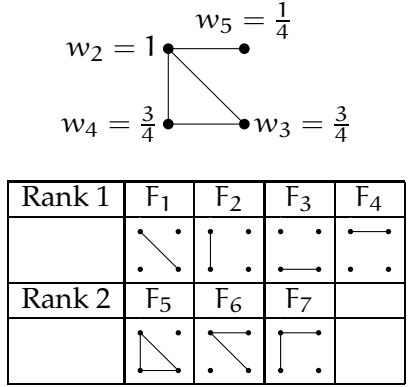
\begin{figure}[ht]
\centering
\begin{tikzpicture}
 \draw (0,0) -- (1,-1) -- (0,-1) -- (0,0);
  \draw (0,0) -- (1,0);
  \fill[black] (0,0) circle (2pt) node[left]{$w_2 = 1$};
  \fill[black] (1,-1) circle (2pt) node[right]{$w_3 = \frac{3}{4}$};
  \fill[black] (0,-1) circle (2pt) node[left]{$w_4 = \frac{3}{4}$};
  \fill[black] (1,0) circle (2pt) node[above]{$w_5 = \frac{1}{4}$};
\end{tikzpicture}\\ \vskip 10pt
\begin{tabular}{|l | c | c | c | c|}
\hline
 Rank 1 & $F_1$ & $F_2$ & $F_3$ & $F_4$ \\ \hline
 & 
 \begin{tikzpicture}[scale = 0.5]    
    \fill[white] (0,0.3) circle (2pt); 
    \fill[black] (0,0) circle (2pt) ;
  \fill[black] (1,-1) circle (2pt) ;
  \fill[black] (0,-1) circle (2pt) ;
  \fill[black] (1,0) circle (2pt) ;
  \draw (0,0) -- (1,-1);
   \end{tikzpicture} &
   \begin{tikzpicture}[scale = 0.5]    
    \fill[white] (0,0.3) circle (2pt); 
    \fill[black] (0,0) circle (2pt) ;
  \fill[black] (1,-1) circle (2pt) ;
  \fill[black] (0,-1) circle (2pt) ;
  \fill[black] (1,0) circle (2pt) ;
  \draw (0,0) -- (0,-1);
   \end{tikzpicture} &
   \begin{tikzpicture}[scale = 0.5]    
    \fill[white] (0,0.3) circle (2pt); 
    \fill[black] (0,0) circle (2pt) ;
  \fill[black] (1,-1) circle (2pt) ;
  \fill[black] (0,-1) circle (2pt) ;
  \fill[black] (1,0) circle (2pt) ;
  \draw (0,-1) -- (1,-1);
   \end{tikzpicture} &
   \begin{tikzpicture}[scale = 0.5]    
    \fill[white] (0,0.3) circle (2pt); 
    \fill[black] (0,0) circle (2pt) ;
  \fill[black] (1,-1) circle (2pt) ;
  \fill[black] (0,-1) circle (2pt) ;
  \fill[black] (1,0) circle (2pt) ;
  \draw (0,0) -- (1,0);
   \end{tikzpicture}
 \\ \hline
 Rank 2 & $F_5$ & $F_6$ & $F_7$ &\\ \hline
 & 
 \begin{tikzpicture}[scale = 0.5]    
    \fill[white] (0,0.3) circle (2pt); 
    \fill[black] (0,0) circle (2pt) ;
  \fill[black] (1,-1) circle (2pt) ;
  \fill[black] (0,-1) circle (2pt) ;
  \fill[black] (1,0) circle (2pt) ;
  \draw (0,0) -- (1,-1) -- (0,-1) -- (0,0);
   \end{tikzpicture} &
   \begin{tikzpicture}[scale = 0.5]    
    \fill[white] (0,0.3) circle (2pt); 
    \fill[black] (0,0) circle (2pt) ;
  \fill[black] (1,-1) circle (2pt) ;
  \fill[black] (0,-1) circle (2pt) ;
  \fill[black] (1,0) circle (2pt) ;
  \draw (1,0) -- (0,0) -- (1,-1);
   \end{tikzpicture} &
   \begin{tikzpicture}[scale = 0.5]    
    \fill[white] (0,0.3) circle (2pt); 
    \fill[black] (0,0) circle (2pt) ;
  \fill[black] (1,-1) circle (2pt) ;
  \fill[black] (0,-1) circle (2pt) ;
  \fill[black] (1,0) circle (2pt) ;
  \draw (1,0) -- (0,0) -- (0,-1);
   \end{tikzpicture} &  
\\\hline 
\end{tabular}

\caption{The reduced weight graph $G( (1,1,3/4,3/4,1/4))$ and its 1-connected flats.} \label{fig_connected_flats_example}
\end{figure}

Figure \ref{fig_nested_set_combinatorial} shows how this can be interpreted as the projection of $\mk{5}$. The latter is the fan over the Petersen graph with 10 rays $v_{ij}$, $i\neq j\subset [5]$ and 15 cones spanned by $v_{\{i,j\}}$ and $v_{\{k,l\}}$ if $\{i,j\}\cap \{k,l\}=\emptyset$.
The projection $\pr_w$ contracts the two rays $v_{\{3,5\}}$ and $v_{\{4,5\}}$, since the corresponding tropical curves are not $w$-stable. The balancing condition around the ray $v_{\{1,2\}}$ is given by the relation $v_{\{1,2\}}=v_{\{3,4\}}+v_{\{3,5\}}+v_{\{4,5\}}$, so we have $\pr_w(v_{\{1,2\}})=\pr_w(v_{\{3,4\}})$, and the 2-dimensional cone spanned by $v_{\{1,2\}}$ and $v_{\{3,4\}}$, even though it corresponds to curves which are $w$-stable, is mapped to a ray. Thus, the fan $\pr_w(\mk{n})$ is not an embedding of the abstract cone complex $\mk{w}$.
We can read off the projection from the graph $G(t)$ interpreted as a subgraph of the complete graph on 4 vertices $K_4$: $\pr_w(v_{\{1,2\}})=\pr_w(v_{\{3,4\}})=v_{F_3}$, $\pr_w(v_{\{1,3\}})=v_{F_7}$, $\pr_w(v_{\{1,4\}})=v_{F_6}$, $\pr_w(v_{\{1,5\}})=v_{F_5}$, $\pr_w(v_{\{2,3\}})=v_{F_1}$, $\pr_w(v_{\{2,4\}})=v_{F_2}$, $\pr_w(v_{\{2,5\}})=v_{F_4}$.

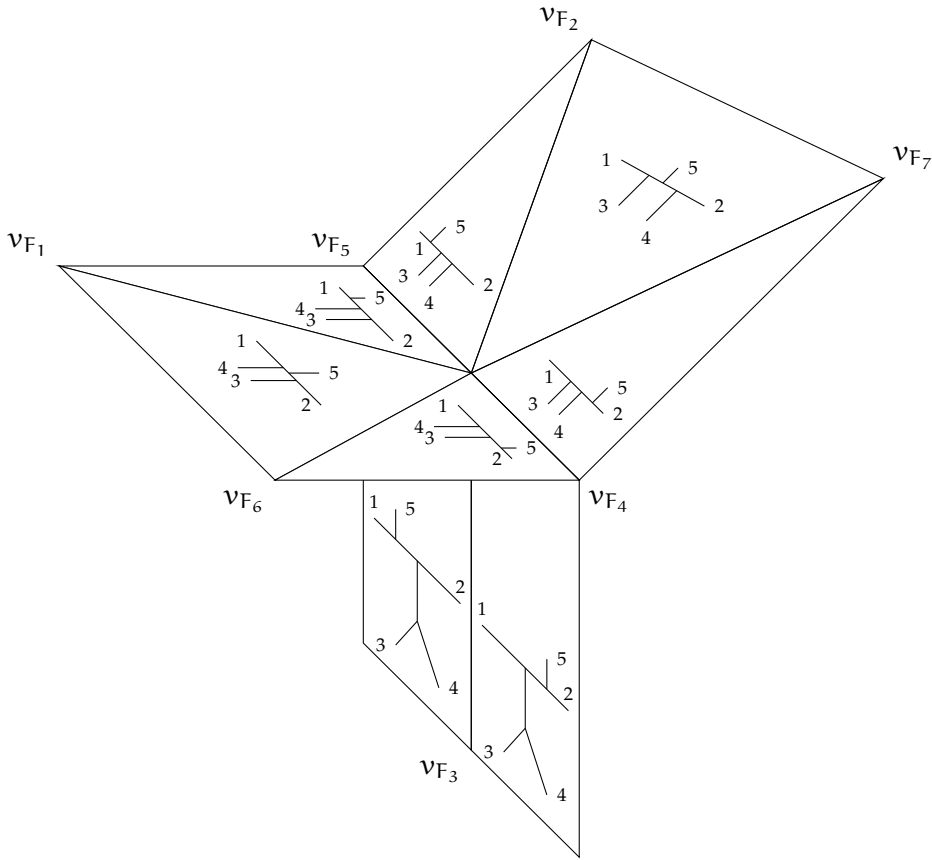
\begin{figure}[ht]
\centering
\begin{tikzpicture}[z = {(-0.71,0.71)},scale = 2]
 \draw (0,0,0) -- (0,0,1) node[above left] {$v_{F_5}$} -- (0,-2.5,1) -- (0,-2.5,0) node[below left]{$v_{F_3}$} -- (0,0,0);
 \draw (0,0,0) -- (0,0,-1) -- (0,-2.5,-1) -- (0,-2.5,0) -- (0,0,0);
 
 \filldraw[fill=white] (0,0,0) -- (0,0,1) -- (-2,0,1) node[above left]{$v_{F_1}$}-- (0,0,0);
 \filldraw[fill=white] (0,0,0) -- (-2,0,1) -- (-2,0,-1) node[below left]{$v_{F_6}$} -- (0,0,0);
 \filldraw[fill=white] (0,0,0) -- (-2,0,-1) -- (0,0,-1)  node[below right]{$v_{F_4}$}-- (0,0,0);
 
 \draw (0,0,0) -- (1.5,1.5,1) node[above left]{$v_{F_2}$} -- (0,0,1) -- (0,0,0);
 \draw (0,0,0) -- (1.5,1.5,1) -- (2,2,-1) node[above right]{$v_{F_7}$}-- (0,0,0);
 \draw (0,0,0) -- (2,2,-1) -- (0,0,-1) -- (0,0,0);
 
 
 \draw (0,-1.6,0.9) node[above]{ \tiny 1}-- (0,-1.6,0.1) node[above]{\tiny 2};
 \draw (0,-1.6,0.7) -- (0,-1.4,0.7) node[right]{\tiny 5};
 \draw (0,-1.6,0.5) -- (0,-2,0.5);
 \draw (0,-2,0.5) -- (0,-2.3,0.7) node[left]{\tiny 3};
 \draw (0,-2,0.5) -- (0,-2.3,0.3) node[right]{\tiny 4};
 \draw (0,-1.6,-0.9) node[above]{ \tiny 2}-- (0,-1.6,-0.1) node[above]{\tiny 1};
 \draw (0,-1.6,-0.7) -- (0,-1.4,-0.7) node[right]{\tiny 5};
 \draw (0,-1.6,-0.5) -- (0,-2,-0.5);
 \draw (0,-2,-0.5) -- (0,-2.3,-0.7) node[right]{\tiny 4};
 \draw (0,-2,-0.5) -- (0,-2.3,-0.3) node[left]{\tiny 3};
 
  \draw (-0.3,0,0.3) node[right]{\tiny 2}-- (-0.3,0,0.8)node[left]{\tiny 1};
 \draw (-0.3,0,0.7) -- (-0.2,0,0.7) node[right]{\tiny 5};
 \draw (-0.3,0,0.5) -- (-0.6,0,0.5) node[left]{\tiny 3};
 \draw (-0.3,0,0.6) -- (-0.6,0,0.6) node[left]{\tiny 4};
 
 \draw (-1.2,0,0.3) node[left]{\tiny 1}-- (-1.2,0,-0.3) node[left]{\tiny 2};
 \draw (-1.2,0,0) -- (-1,0,0) node[right]{\tiny 5};
 \draw (-1.2,0,0.05) -- (-1.5,0,0.05) node[left]{\tiny 4};
 \draw (-1.2,0,-0.07) -- (-1.5,0,-0.07) node[left]{\tiny 3};
 
 \draw (-0.3,0,-0.3) node[left]{\tiny 1}-- (-0.3,0,-0.8)node[left]{\tiny 2};
 \draw (-0.3,0,-0.7) -- (-0.2,0,-0.7) node[right]{\tiny 5};
 \draw (-0.3,0,-0.5) -- (-0.6,0,-0.5) node[left]{\tiny 4};
 \draw (-0.3,0,-0.6) -- (-0.6,0,-0.6) node[left]{\tiny 3};
 
 \draw (0.3,0.3,0.9) node[below]{\tiny 1}-- (0.3,0.3,0.4) node[right]{\tiny 2};
 \draw (0.3,0.3,0.7) -- (0.15,0.15,0.7) node[left]{\tiny 3};
 \draw (0.3,0.3,0.6) -- (0.15,0.15,0.6) node[below]{\tiny 4};
 \draw (0.3,0.3,0.8) --(0.4,0.4,0.8) node[right]{\tiny 5};
 
 \draw (1.2,1.2,0.3) node[left]{\tiny 1}-- (1.32,1.32,-0.3) node[right]{\tiny 2};
 \draw (1.24,1.24,0.1) -- (1.04,1.04,0.1) node[left]{\tiny 3};
 \draw (1.28,1.28,-0.1) -- (1.08,1.08,-0.1) node[below]{\tiny 4};
 \draw (1.26,1.26,0) -- (1.36,1.36,0) node[right]{\tiny 5};
 
  \draw (0.3,0.3,-0.8) node[right]{\tiny 2}-- (0.3,0.3,-0.3) node[below]{\tiny 1};
 \draw (0.3,0.3,-0.6) -- (0.15,0.15,-0.6) node[below]{\tiny 4};
 \draw (0.3,0.3,-0.5) -- (0.15,0.15,-0.5) node[left]{\tiny 3};
 \draw (0.3,0.3,-0.7) --(0.4,0.4,-0.7) node[right]{\tiny 5};

\end{tikzpicture}
 \caption{The nested set subdivision of $B'(G(1,1,3/4,3/4,1/4))$ and its combinatorial interpretation.}\label{fig_nested_set_combinatorial}
\end{figure}
This suggests an interpretation of the combinatorics of $B'(G(w))$ as follows (see~Figure \ref{fig_nested_set_combinatorial}). Let $C$ be an element in $\mk{4}$, i.e.\ a four-marked tropical curve with labels $\{1,\dots,4\}$ and weights $(1,1,3/4,3/4)$. We can interpret the additional $\R$-coordinate as placing the leaf $5$ with weight 1/4 somewhere along the subgraph consisting of leaves 1 and 2 and the edge between them, if it exists. The subdivision of the cones of $\mk{4} \times \R$ given by the building set of $1$-connected flats is then obtained by subdividing a cone if the attached leaf $5$ is at a trivalent vertex. We will see in Section \ref{section_fibre} that we can always interpret $B'(G(w))$ in this fashion.

We have seen that $\pr_w(\mk{n})$ is not an embedding of $\mk{w}$ as a balanced fan, since it is ``missing'' the cone $\sigma$ spanned by $v_{\{1,2\}}$ and $v_{\{3,4\}}$. The codimension $1$ type $v_{\{3,4,5\}}$ is $w$-stable, but is adjacent to only $1$ $w$-stable maximal cell, namely $\sigma$. We obtain a ``univalent'' codimension one face. There is no way to embed this codimension one face and its adjacent cones into a vector space as a balanced fan (\textit{cf.} Figure \ref{moduli_fig_onlyonestable}).

 \begin{figure}[ht]
  \centering
  \begin{tikzpicture}
   \matrix[column sep = 30pt,row sep = 10pt, ampersand replacement=\&]{
    \& \& 
    \draw (0,0) -- (2,0); 
    \draw (0,0) -- (-0.5,0.3) node[left]{1};
    \draw (0,0) -- (-0.5,-0.3) node[left]{2};
    \draw (1,0) -- (1,0.5) node[above]{5};
    \draw (2,0) -- (2.5,0.3) node[right]{3};
    \draw (2,0) -- (2.5,-0.3) node[right]{4};
    \draw (2,0) node[right = 30pt] {stable};
    \\
    \draw (0,0) -- (1,0); 
    \draw (0,0) -- (-0.5,0.3) node[left]{1};
    \draw (0,0) -- (-0.5,-0.3) node[left]{2};
    \draw (1,0) -- (1.5,0.3) node[right]{3};
    \draw (1,0) -- (1.5,0) node[right]{4};
    \draw (1,0) -- (1.5,-0.3) node[right]{5};
    \& \draw[->] (0,0) -- (1,0); \& 
    \draw (0,0) -- (2,0); 
    \draw (0,0) -- (-0.5,0.3) node[left]{1};
    \draw (0,0) -- (-0.5,-0.3) node[left]{2};
    \draw (1,0) -- (1,0.5) node[above]{3};
    \draw (2,0) -- (2.5,0.3) node[right]{4};
    \draw (2,0) -- (2.5,-0.3) node[right]{5};
    \draw (2,0) node[right = 30pt] {not stable};
    \\
    \& \& 
    \draw (0,0) -- (2,0); 
    \draw (0,0) -- (-0.5,0.3) node[left]{1};
    \draw (0,0) -- (-0.5,-0.3) node[left]{2};
    \draw (1,0) -- (1,0.5) node[above]{4};
    \draw (2,0) -- (2.5,0.3) node[right]{3};
    \draw (2,0) -- (2.5,-0.3) node[right]{5};
    \draw (2,0) node[right = 30pt] {not stable};
    \\
   };
  \end{tikzpicture}
  \caption{For $w = (1,1,3/4,3/4,1/4)$, only one of the maximal cones adjacent to the codimension one type $v_{\{3,4,5\}}$ corresponds to a stable combinatorial type. We can therefore not embed $\mk{w}$ as a balanced fan.}\label{moduli_fig_onlyonestable}
 \end{figure}
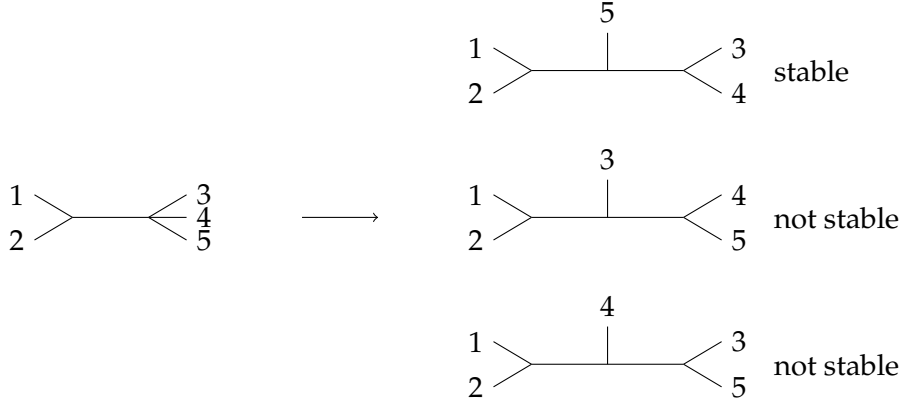
 \end{ex}
 
The following definition characterizes those cones that obstruct a balanced embedding of $\mk{w}$ in a vector space. Such cones will correspond precisely to the top-dimensional cones of $\mk{n}$ on which $\pr_w$ is not injective, see Lemma \ref{moduli_thm_noninjective}.

\begin{defn}
 Let $n \geq 4$ and $w$ be a weight vector. We consider $\mk{n}$ in its combinatorial subdivision.
 \begin{itemize}
  \item We denote by $\curly{U}_w^0$ the collection of all top-dimensional cones $\sigma$ of $\mk{n}$ such that the corresponding combinatorial type is not $w$-stable.
  \item We recursively define $\curly{U}_w^{k+1}$: A top-dimensional cone $\sigma \notin \bigcup_{j=0}^k \curly{U}_w^j$ is in $\curly{U}_w^{k+1}$ if and only if it has a codimension one face $\tau$ such that all other top-dimensional cones $\sigma'$ neighboring $\tau$ lie in $\bigcup_{j=0}^k \curly{U}_w^j$.
 \end{itemize}
Finally we set $\curly{U}_w := \bigcup_{k \geq 0} \curly{U}_w^k$. We call curves lying in the support of $\curly{U}_w$ \textit{inherited $w$-unstable}.

\end{defn}

\begin{lemma}\label{moduli_lemma_nested}
 Let $\sigma$ be a top-dimensional cone of $\mk{n}$ with rays $v_{I_1},\dots,v_{I_{n-3}}$ and assume $1 \notin I_j$ for all $j$. Let $F_{I_k}$ be the flat of $K_{n-1}$ corresponding to the complete graph on vertices in $I_k$. Then 
 $$\{F_{I_1} \cap G(w),\dots,F_{I_{n-3}} \cap G(w)\}$$
 is a nested set with respect to the building set $\curly{G}_{G(w)}$ of 1-connected flats in $G(w)$.
 \begin{proof}
  Let $C_\sigma$ be the combinatorial type associated to $\sigma$. Then $I_1,\dots,I_k$ are the leaf splits induced by the bounded edges of $C_\sigma$. Since we assumed $1 \notin I_j$ for all $j$, we have that any two incomparable $I_i,I_j$ are already disjoint. In particular, any two incomparable flats $F_{I_i} \cap G(w), F_{I_j} \cap G(w)$ must be vertex-disjoint. Now notice that in any graph, the join of two vertex-disjoint flats is just the union. As this is not a connected graph, the claim follows. 
 \end{proof}
\end{lemma}

\begin{lemma}\label{moduli_lemma_wheninjective}
 Let $\sigma$ be a top-dimensional cone of $\mk{n}$ such that $\pr_w$ is not injective on $\sigma$. Then one of the following holds:
 \begin{itemize}
  \item $\sigma$ has a ray $r$ such that $\pr_w(r) = 0$.
  \item $\sigma$ has rays $r,s$ such that $\pr_w(r) = \pr_w(s)$.
 \end{itemize}
 \begin{proof}
  Assume $\sigma$ has rays $v_{I_1},\dots,v_{I_{n-3}}$, which are mapped to non-zero, distinct elements $v_{F_{I_j} \cap G(w)}$. By Lemma \ref{moduli_lemma_nested}, the flats $F_{I_j} \cap G(w)$ form a nested set. It is well-known that nested set subdivisions are simplicial, so the corresponding vectors must be linearly independent. In particular $\pr_w$ must be injective on $\sigma$.  
 \end{proof}
\end{lemma}

\begin{lemma}\label{moduli_thm_noninjective}
 Let $n \geq 4$ and $w $ a weight vector as before. Let $\sigma$ be a top-dimensional cone of $\mk{n}$ in its combinatorial subdivision. Then $\sigma \in \curly{U}_w$ if and only if $\pr_w$ is not injective on $\sigma$. 
\end{lemma}
 \begin{proof}
  First, let $\sigma \in \curly{U}_w = \bigcup_{k \geq 0} \curly{U}_w^k$. We will prove that $\pr_w$ is not injective on $\sigma$ by induction on $k$.  Assume $\sigma \in \curly{U}_w^0$. Then the combinatorial type of $\sigma$ is not $w$-stable, so $\sigma$ has a ray of the form $v_{\{i,j\}}$ with $w_i + w_j \leq 1$. This is equal to the ray $v_{F_{ij}}$, where $F_{ij}$ is the flat of $K_{n-1}$ consisting only of the edge between nodes $i$ and $j$. This edge does not exist in $G(w)$, so $\pr_w(v_{F_{ij}}) = 0$.
  
Consider a cone $\sigma \in \curly{U}_w^{k+1}$. That is, there is a codimension one cone $\tau$ such that the top-dimensional cones adjacent to $\tau$ are $\sigma, \sigma', \sigma''$ with $\sigma', \sigma'' \in \curly{U}_w^k$. By induction, we may suppose that $\pr_w$ is not injective on $\sigma'$ or $\sigma''$. The projection morphism induces a morphism on the local fan $\Star_{\mk{n}}(\tau)$, which is a tropical line, i.e.\ a one-dimensional balanced fan with three rays. The rays corresponding to $\sigma'$ and $\sigma''$ are mapped to 0, so by linearity of the local morphism, the ray corresponding to $\sigma$ must also be mapped to 0. Hence $\pr_w$ is not injective on $\sigma$.
  
Suppose that $\pr_w$ is not injective on a cone $\sigma$ with rays $v_{I_1},\dots,v_{I_{n-3}}$. By Lemma \ref{moduli_lemma_wheninjective}, we notice that either one ray of $\sigma$ is mapped to 0, or two rays are mapped to the same element. In the former case, if $\pr_w(v_{I_j}) = 0$ for some $j$, then $w_a + w_b \leq 1$ for all $a,b \in I_j$. But this implies that $\sigma \in \curly{U}_w^0$. We now consider the latter possibility, where two rays are mapped to the same ray,  say $\pr_w(v_{I_i}) = \pr_w(v_{I_j})$ for some $i \neq j$. We may assume that $1 \notin I_i,I_j$. Then we can also assume that $I_j \subset I_i$, since otherwise $I_i \cap I_j = \emptyset$ and thus $\pr_w(v_{I_i}) = \pr_w(v_{I_j}) = 0$. Now $I_i$ and $I_j$ correspond to two edges. Assume these edges do not share a vertex. Then there must be a chain of edges connecting them, corresponding to splits $I_j  = J_1 \subset \dots \subset J_t = I_i$. As $\pr_w(I_i) = \pr_w(I_j)$, we must have $w_k + w_l \leq 1$ for all $k \in I_i \wo I_j, l \in I_j$. In particular $\pr_w(J_s) = \pr_w(I_j)$ for all $s = 1,\dots,t$. Hence we can assume that the edges corresponding to $I_i$ and $I_j$ share a common vertex. Denote these edges by $e_i$ and $e_j$ and the vertex by $v_{ij}:= e_i \cap e_j$.

  First let us assume that $|I_i| - |I_j| = 1$, i.e.\ there is an additional leaf $l$ at $v_{ij}$. We prove that this cone is inherited unstable by an induction on $|I_j|$. We start with $|I_j| = 2$, i.e.\ $I_j = \{a,b\}$. This implies $w_l + w_a, w_l + w_b \leq 1$. We obtain a codimension one type $C_\tau$ by contracting the edge $e_j$. The two other adjacent top-dimensional types besides $C_\sigma$ have rays $v_{\{l,i\}}$ and $v_{\{l,j\}}$ respectively, both of which are mapped to 0. In particular, $\sigma$ lies in $\curly{U}_w^1$. Now assume $|I_j| > 2$. The vertex of $e_j$ which is not $v_{ij}$ is also trivalent, i.e. we have a partition of $I_j$ into two edge splits $I_j', I_j''$. Again, we obtain a codimension one type $C_\tau$ by contracting $e_j$. We obtain an adjacent top-dimensional type $C_{\sigma'}$ replacing the rays $v_{I_i},v_{I_j}$ with rays $v_{I_j' \cup \{l\}}, v_{I_i}$. 
We now argue by induction that $\sigma' \in \curly{U}_w$. The exact same argument works for the third maximal cone $\sigma'' > \tau$, which finally implies $\sigma\in \curly{U}_w$ as required.

If $I_j' = \{a\}$, then $w_l + w_a \leq 1$ and $\sigma' \in \curly{U}_w^0$. If $|I_j'
| \geq 2$, then $\pr_w(v_{I_j'}) = \pr_w(v_{I_j' \cup \{l\}})$ (as $l$ is not connected to any element of $I_j$ in $G(w)$). As $|I_j'| < |I_j|$ and $\abs{I_j'\cup \{l\}}- \abs{I_j'}=1$, we can apply induction to see that $\sigma' \in \curly{U}_w$ (see also Figure \ref{moduli_fig_induction}). 

  An analogous induction argument may be used in the case that $|I_i| - |I_j| > 1$.
\end{proof}

  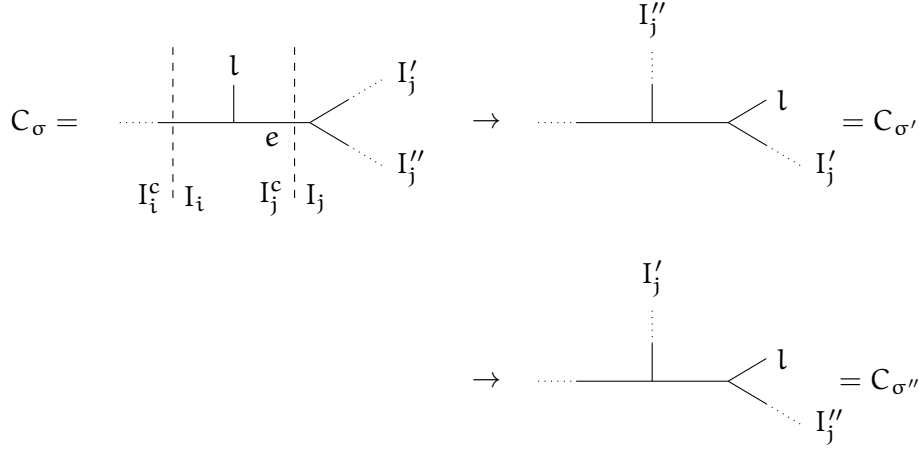
\begin{figure}[ht]
   \centering
    \begin{tikzpicture}
     \matrix[ampersand replacement = \&, column sep = 10pt, row sep = 10pt]{
	\draw[dotted] (-0.5,0) -- (0,0);
	\draw (0,0) -- (2,0);
	\draw (2,0) -- (2.5,0.3);
	\draw[dotted] (2.5,0.3) -- (3,0.6) node[right]{$I_j'$};
	\draw (2,0) -- (2.5,-0.3);
	\draw[dotted] (2.5,-0.3) -- (3,-0.6) node[right]{$I_j''$};
	\draw (1,0) -- (1,0.5) node[above]{$l$};
	\draw (-1.5,0) node{$C_\sigma = $};
	\draw (1.5,0) node[below]{$e$};
	\draw[dashed] (1.8,-1)node[right]{$I_j$} node[left]{$I_j^c$} -- (1.8,1) ;
	\draw[dashed] (0.2,-1) node[right]{$I_i$} node[left]{$I_i^c$}-- (0.2,1) ;

	    \& \draw (0,0) node{$\rightarrow$}; \&
	\draw[dotted] (-0.5,0) -- (0,0);
	\draw (0,0) -- (2,0);
	\draw (2,0) -- (2.5,0.3)node[right]{$l$};
	\draw (2,0) -- (2.5,-0.3);
	\draw[dotted] (2.5,-0.3) -- (3,-0.6) node[right]{$I_j'$};
	\draw (1,0) -- (1,0.5);
	\draw[dotted] (1,0.5) -- (1,1) node[above]{$I_j''$};
	\draw (4,0) node{$= C_{\sigma'}$};
	 \\
	\& 
	\draw (0,0) node{$\rightarrow$};\&
	\draw[dotted] (-0.5,0) -- (0,0);
	\draw (0,0) -- (2,0);
	\draw (2,0) -- (2.5,0.3)node[right]{$l$} ;
	\draw (2,0) -- (2.5,-0.3);
	\draw[dotted] (2.5,-0.3) -- (3,-0.6) node[right]{$I_j''$};
	\draw (1,0) -- (1,0.5);
	\draw[dotted] (1,0.5) -- (1,1) node[above]{$I_j'$};
	\draw (4,0) node{$= C_{\sigma''}$};
	\\
      };
    \end{tikzpicture}
    \caption{Assuming $\pr_w(v_{I_j}) = \pr_w(v_{I_i})$ in $\sigma$, we can see that the cones adjacent to the codimension one type obtained by shrinking e are already in $\curly{U}_w$ by using induction.}\label{moduli_fig_induction}
  \end{figure}

\begin{corollary}\label{moduli_cor_needstable}
 Let $w $ be a weight vector. Then $\pr_w$ contracts all top-dimensional cones of $\mk{n}$ if and only if $w$ does not have at least two heavy entries.
 \end{corollary}
 \begin{proof}
  First, assume $w$ has no heavy entry. Choose $j \in [n]$, such that $w_j$ is minimal. Then $w_j + w_k \leq 1 $ for all $k \neq j$: If we assume $w_j + w_k > 1$ for some $k$, then $w_k + w_l > 1$ for all $ l\neq k$ and $k$ is heavy in $w$.
  
In particular, in any trivalent $w$-stable combinatorial type the leaf $j$ can only be at a vertex with two bounded edges $e,e'$ corresponding to rays $v_I, v_{I \cup \{j\}}$. But both rays are mapped to the same image under $\pr_w$. 
   
 Now suppose $w_1$ is the only heavy weight. Again, choose $j$ such that $w_j$ is minimal. As before, we must have $w_j + w_k \leq 1$ for all $k \neq 1,j$. Let $C$ be a trivalent, $w$-stable curve. If leaf $j$ is attached to a vertex with two bounded edges, these edges correspond to rays $v_I, v_{I \cup \{j\}}$ with $1 \notin I$, which are mapped to the same image. In particular, the cone corresponding to $C$ lies in $\curly{U}_w$ by Lemma \ref{moduli_thm_noninjective}. Hence $C$ must have a bounded edge corresponding to the ray $v_{\{1,j\}}$. If we contract this edge, we obtain a curve corresponding to a codimension one cone. But any other adjacent top-dimensional types are in $\curly{U}_w$: It either has a ray $v_{\{j,k\}}, k \neq j$, which makes it unstable, or leaf $j$ is adjacent to two bounded edges. The corresponding cone lies in $\curly{U}_w$ by our previous argument. Hence the cone of $C$ lies in $\curly{U}_w$ as well and thus is contracted under $\pr_w$ by Lemma \ref{moduli_thm_noninjective}.
   
   Conversely, assume $w$ has heavy entries $i$ and $j$. Then we can easily construct a trivalent curve whose cone does not lie in $\curly{U}_w$: Let $\{i_1,\dots,i_{n-2}\} = [n]\wo \{i,j\}$ be in some arbitrary but fixed order. Then it is easy to see that $\pr_w$ is injective on the cone corresponding to the \textit{caterpillar tree} (see Figure~\ref{geom_fig_losevmanin}) with edges $v_{\{i,i_1\}},v_{\{i,i_1,i_2\}}, \dots,v_{\{i,\dots,i_{n-3}\}}$.
   
 \end{proof}

We now have the necessary ingredients to prove Theorem \ref{moduli_cor_mainresult}.

 \begin{proof}[Proof of Theorem \ref{moduli_cor_mainresult}]
  The dimension of $B'(G(w))$ equals $\textnormal{rank}(G(w)) -1$. By assumption, $G(w)$ is a connected graph, so its rank is just the number of its vertices minus one. In total, we obtain $\dim B'(G(w)) = n-3 = \dim \mk{n}$. By Corollary \ref{moduli_cor_needstable}, the dimension of the image of $\pr_w$ is also $n-3$. Since $B'(G(w))$ is irreducible, $\pr_w$ is surjective. Using Lemma \ref{moduli_lemma_nested}, it is easy to see that two different top-dimensional cones on which $\pr_w$ is injective are mapped to distinct cones of the nested set subdivision of $B'(G(w))$.
 \end{proof}

We now study balanced fan structures on $\mk{w}$. 


\begin{prop}\label{prop-nocontract}
 Let $w $ be a weight vector with at least two heavy entries. Then $\curly{U}_w = \curly{U}_w^0$ if and only if every $i \in [n]$ is either heavy or small in $w$. 
 \begin{proof}
  First, assume that $\curly{U}_w = \curly{U}_w^0$. Assume there is an entry $i$, which is neither small nor heavy in $w$. Hence, there must be a $j$, such that $w_i + w_j > 1$, but $j$ is not heavy (in particular, it is also not small in $w$). It follows that there must be a $k \neq i,j$, such that $w_i + w_k, w_j + w_k \leq 1$. 
  
  For now assume $n \geq 5$. Let $a,b$ be two heavy entries of $w$ and fix an order $\{i_1,\dots,i_{n-5}\}$ on $[n] \wo \{i,j,k,a,b\}$. Now let $\sigma$ be the cone with rays $$v_{\{i,j\}}, v_{\{i,j,k\}},v_{\{i,j,k,i_1\}}\dots,v_{\{i,j,k,i_1,\dots,i_{n-5}\}}$$ (see also Figure \ref{moduli_fig_stillunstable}). The corresponding combinatorial type is clearly $w$-stable, but $\pr_w(v_{\{i,j\}}) = \pr_w(v_{\{i,j,k\}})$. This is a contradiction.
  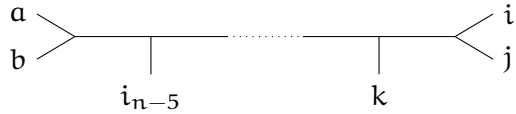
\begin{figure}[ht]
   \centering
   \begin{tikzpicture}
    \draw (0,0) -- (2,0);
    \draw[dotted] (2,0) -- (3,0);
    \draw (3,0) -- (5,0);
    \draw (0,0) -- (-0.5,0.3) node[left]{$a$};
    \draw (0,0) -- (-0.5,-0.3) node[left]{$b$};
    \draw (1,0) -- (1,-0.5) node[below]{$i_{n-5}$};
    \draw (4,0) -- (4,-0.5) node[below]{$k$};
    \draw (5,0) -- (5.5,0.3) node[right]{$i$};
    \draw (5,0) -- (5.5,-0.3) node[right]{$j$};
   \end{tikzpicture}
   \caption{Constructing a $w$-stable type that lies in $U_w$.}\label{moduli_fig_stillunstable}
  \end{figure}
  
  If $n = 4$, then only the cone spanned by $v_{\{i,j\}}$ corresponds to a $w$-stable type, so by definition it must lie in $\curly{U}_w$.
  
  To see the converse, we can assume that $w$ is of the form $(1^f, \epsilon^t)$ with $f \geq 2$. It suffices to show that $\curly{U}_w^1 = \emptyset$, i.e.\ after removing all $w$-unstable types, all remaining codimension one combinatorial types still have at least two resolutions. This is clear and the claim follows immediately.
 \end{proof}
\end{prop}

\begin{theorem}\label{thm-heavyandlight}
Let $w$ be heavy/light, with at least two heavy entries. The cone complex underlying $\pr_w(\mk{n})=B'(G(w))$ is naturally identified with $M_{0,w}^{\trop}$. In particular, this complex has the structure of a balanced fan. If $w$ is not heavy/light, then there does not exist a balanced embedding of $\mk{w}$ into a vector space.
\end{theorem}

\begin{proof}
Suppose $w$ is a weight vector with only heavy and light entries. It follows from Lemma \ref{moduli_thm_noninjective} that $\pr_w$ is injective on all cones which are not in $\curly{U}_w$. We can deduce from Proposition~\ref{prop-nocontract} that $\curly{U}_w=\curly{U}_w^0$. In other words, $\pr_w$ contracts only the top-dimensional cones we want to contract to pass from $\mk{n}$ to $\mk{w}$. The fact that the small points are light guarantees that $\mk{w}$ is pure dimensional. To see this, consider a cone whose top dimensional faces are contracted. Such a cone is spanned by vectors $v_{I_j}$, and there must exist at least one vector such that all $k\in I_j$ are light. Analyzing the projections, we see that $B'(G(w))$ is identified with $M^{\trop}_{0,w}$. 

Consider a weight vector $w$ that is not of heavy/light type. We first deal with the case where $w$ has heavy and small points. Recall $i$ is small if $w_i+w_j>1$ implies that $j$ is heavy. In this situation, there is a subset $I$ of size at least three of small points satisfying $\sum_{k\in I}w_k>1$, but for any subset $I_0\subsetneq I$, $\sum_{k\in I_0}w_k<1$. Observe that there is a cone $\tau$ of $\mk{w}$ of maximal possible dimension containing the ray $v_I$ such that all its higher-dimensional faces of $\tau$ in $\mk{n}$ are not $w$-stable. 

Since we have at least two heavy points, the top-dimensional cone of $\mk{n}$ corresponding to the caterpillar tree with the two heavy weights on the two sides as in the proof of Corollary \ref{moduli_cor_needstable} is $w$-stable, and hence is also a top dimensional cone of $\mk{w}$. We can now see that $\mk{w}$ is not pure dimensional, and cannot be embedded as a balanced fan.
Assume now that $w$ has not only heavy and small points. By Proposition  \ref{prop-nocontract}, $\curly{U}_w\neq \curly{U}_w^0$, so there are cones of codimension one with only one adjacent top-dimensional cone. These cones cannot be embedded in a balanced way.
\end{proof}

\section{Tropicalizing spaces of rational weighted stable curves}\label{sec: trop-wsc}
Throughout, we work in the ``constant coefficient'' case, i.e. over $\mathbb C$ with the trivial valuation. We have seen that for a vector of heavy and light weights, we obtain a fan structure for the tropical moduli space $\mk{w}$. In this section, we  show that this fan yields a toric variety in which we can embed $\overline{M}_{0,w}$, and the tropicalization of the open part living inside the torus can be given a canonical fan structure, making it isomorphic to $\mk{w}$. If we have not only heavy and light weights, then we still have a map from $\overline{M}_{0,w}$ to the toric variety defined by $\pr_w(\mk{n})$, but it contracts some boundary strata.

\subsection{Geometric tropicalization for \texorpdfstring{$\overline M_{0,n}$}{M0,n}}
In \cite[Example 3.1]{GM07}, the locus of smooth curves $M_{0,n}$ is identified with the quotient of an open set of the Grassmannian, denoted $G^0(2,n)$, by an $n-1$ dimensional torus. The open set $G^0(2,n)$ corresponds to the $2$ planes that do not pass through the intersection of a pair of coordinate planes.

The Grassmannian $G(2,n)$ embeds into $\P^{\binom{n}{2}-1}$ via the Pl\"ucker embedding, associating to a $2$-plane given by the data of a $2\times n$ matrix (after a choice of basis) its minors. This embedding carries the open part $G^0(2,n)$ to points in the torus of $\P^{\binom{n}{2}-1}$. As a consequence, $M_{0,n}$ is embedded into the torus $(T^{\binom{n}{2}}/T)/T^{n-1}\cong T^{\binom{n}{2}-n}$ using the Pl\"ucker embedding. 

Note that the action of the $T^{n}$ torus corresponds precisely to the lineality space $\Phi(\R^n)$ that we quotient by embedding $\mk{n}$ into $\R^{\binom{n}{2}-n}$ (see Section \ref{prelim-trop}). Comparing coordinates on the algebraic and tropical sides, we can effectively neglect the action of $T^{n}$ on one side and the lineality space on the other. Furthermore, the Pl\"ucker coordinates give us the distance coordinates in tropical geometry directly.

Keeping with the discussion in previous sections, we recall that the tropical moduli space $M_{0,n}^{\trop}$ comes with an embedding in a vector space, and a natural fan structure. Fix this fan structure. Then we have the following result, due to Gibney and Maclagan~\cite[Theorem 5.7]{GM07}, as well as Tevelev~\cite[Theorem 5.5]{Tev07}
\begin{theorem}
Consider the embedding of $M_{0,n}$ into the torus $T^{\binom{n}{2}-n}$ described above. The closure of $M_{0,n}$ in the toric variety $X(M_{0,n}^{\trop})$ is the 
compactification $\overline M_{0,n}$. Furthermore, the tropicalization of $M_{0,n}$ in this torus is $\mk{n}$.
\end{theorem}

The essential ingredient in the proof of this result is the understanding of the combinatorial structure of $M^{\trop}_{0,n}$ and in particular its relationship to the boundary stratification of the classical moduli spaces $\overline M_{0,n}$. The results provide a beautiful explanation of the various analogies and combinatorial dualities between the tropical and classical moduli space in question. In Theorem~\ref{thm-tropicalizing}, we obtain analogous results for spaces of weighted stable curves. 

The result can be obtained with the help of a technique that is now known as \textit{geometric tropicalization} -- initially used to study compactifications of subvarieties of tori in~\cite{Tev07}. The technique was elaborated upon and applied to understand compactifications of moduli space of del Pezzo surfaces in~\cite{HKT09}. An accessible introduction to the topic can be found in \cite{Cue11, MS09}.

Geometric tropicalization starts with a variety $X$ together with a simple normal crossing boundary divisor $D$ (such as $\overline{M}_{0,n}$ with its usual boundary). When the complement $U$ of $D$ in $X$ has many invertible functions, it admits a map to a torus: 
\[
\iota: U\to T.
\]
In ideal situations (and indeed, in our situation) this map is an embedding. The map $\iota$ may then be used to produce a map from the dual intersection complex $\Sigma$ of $(X, D)$ to the vector space of one parameter subgroups of $T$, thus furnishing a fan structure on $\Sigma$.

In~\cite{Cue11}, Cueto equips the top dimensional cones of $\Sigma$ with a  weight function which produces a balanced fan. This fan furnishes a toric variety with dense torus $T$, in which we may consider the closure of $\iota(U)$. Geometric tropicalization studies the relationship between $X$ and $\overline {\iota(U)}$. 


\subsection{Geometric tropicalization for spaces of weighted stable curves}\label{subsec: geom-trop}
We assume that $w$ is a weight vector with only heavy and light weights, that there are at least $2$ heavy weights, and without loss of generality, that the heavy entries include the first two. In other words, throughout this section, we fix $w = (1^f,\epsilon^t)$ where $\epsilon$ is light. We have seen already that the assumption of having at least two heavy weights makes sense in tropical geometry, as without it, the tropical moduli space is of incorrect dimension. On the algebraic side, this requirement is also natural for weight vectors of the form $w = (1^f,\epsilon^t)$, since otherwise, the locus of smooth curves in $\overline M_{0,w}$ would be empty.

We apply the geometric tropicalization techniques discussed above in the context of heavy/light moduli spaces of rational pointed curves. Note that the compactification of $M_{0,n}$ to $\overline{M}_{0,w}$ is \textit{not}, in general, simple normal crossing and may be locally a more complicated hyperplane arrangement. The following observation is crucial, and its proof is identical to the $\overline M_{0,n}$ case. An elegant proof in this setting may be found in~\cite[Theorem 1.1]{U}.

\begin{prop}
Let $M_{0,w}$ denote the locus of smooth curves in $\overline M_{0,w}$. Then, the boundary $\overline M_{0,w}\setminus M_{0,w}$ is a divisor with simple normal crossings.
\end{prop}

\begin{remark}
In other words, we consider the ``interior'' of the moduli space $\overline M_{0,w}$ as not only points of $M_{0,n}$ but also loci of marked curves where the underlying curve is smooth, but the markings are not distinct. We warn the reader that the Hassett spaces $\overline{\mathcal{M}}_{g,w}$ are usually \textit{not} toroidal compactifications of $\mathcal{M}_{g,w}$, however, the locus of non-smooth curves is always a divisor with (stacky) normal crossings.
\end{remark}

\begin{remark}~\label{rem: m0w-boundary}
The inclusion $M_{0,w}\hookrightarrow \overline M_{0,w}$ induces a stratification into locally closed strata, which coincides with the stratification by dual graph: the codimension $k$ strata of $\overline M_{0,w}$ are the loci of curves, with fixed dual graph, having $k$ nodes. Locally analytically near a stratum $S$, there is a collection of monomial coordinates on $\overline M_{0,w}$ given by the deformation parameters for the nodes of the curves parametrized by $S$. 
\end{remark}

\begin{defn}
 Let $\Pr_w$ be the projection from the torus $T^{\binom{n}{2}}/T^n$ dropping all the Pl\"ucker coordinates indexed by $i\neq j\in [n]$  for which $w_i=w_j=\epsilon$.
\end{defn}

\begin{lemma}\label{lem-prw}
 The tropicalization of the map $\Pr_w$ agrees with the projection $\pr_w$ from $\R^{\binom{n}{2}-n}$ (see Remark \ref{rem-prw}).
\end{lemma}
\begin{proof}
By \cite{KM07}, Lemma 2.3 and 2.4, the vectors $v_I$ where $I$ is a two-element subset not containing $1$ and not equal to $\{2,3\}$ form a basis of $\R^{\binom{n}{2}}/\Phi(\R^n) $.
A ray $v_I$ can be expressed in terms of the basis vectors using \cite{KM07}, Lemma 2.6, which tells us that $v_I$ equals the sum of all $v_S$ where $S\subset I$ is a two-element subset (we assume without restriction that $1\notin I$), and the fact that $-v_{\{2,3\}}$ equals the sum of our basis vectors above.
The tropicalization of the map $\Pr_w$ contracts the vectors $v_{\{i,j\}}$ (which equal $-2e_{ij}$ modulo the lineality space, being the images of $-e_{ij}$ in $B'(K_{n-1})$) for all $i,j$ such that $w_i=w_j=\epsilon$, and with these, it also contracts all rays of the form $v_I$ with $w_i=\epsilon$ for all $i\in I$, since we can express the latter in terms of the $v_{\{i,j\}}$ by the above. Thus, it equals the map $\pr_w$.
\end{proof}

\begin{lemma}\label{lem-embed}
The open part $M_{0,w}$ can be embedded into the torus $\Pr_w( T^{\binom{n}{2}}/T^n)$ using the Pl\"ucker coordinates.
\end{lemma}
\begin{proof}
 Let us compare the open part $M_{0,w}$ to $M_{0,n}$: now points which are both light are allowed to collide. In the $2\times n$ matrix describing a collection of $n$ points (resp.\ a two-plane in $G(2,n)$) this means that two columns can now coincide (up to nonzero multiple), leading to a zero minor. However, these are exactly the minors we project away with $\Pr_w$, so using the remaining Pl\"ucker coordinates, we embed $M_{0,w}$ into the torus $\Pr_w( T^{\binom{n}{2}}/T^n)$.
\end{proof}

\begin{lemma}\label{lem-geomtrop}
 The geometric tropicalization of $\overline{M}_{0,w}$ using the embedding in Lemma \ref{lem-embed} is identified with $\pr_w(\mk{n})$.
\end{lemma}
\begin{proof}
It is straightforward to see that the cone over the dual intersection complex of $\overline{M}_{0,w}$ is canonically identified (as a cone complex) with $\mk{w}$. We know already by Theorem \ref{thm-heavyandlight} that the latter is the cone complex underlying the fan $\pr_w(\mk{n})$.
Thus, it remains only to check that the divisorial valuations of the boundary divisors for the remaining Pl\"ucker coordinate functions yield the rays of this fan. This is an easy consequence of Lemma \ref{lem-prw}.

Finally, we must check that the weight function for the geometric tropicalization, as given in \cite[Theorem 2.5]{Cue11}, is identically $1$, thus matching the weight on the Bergman fan $\pr_w(\mk{n})=B'(G(w))$. This again follows from the analogous fact for $\mk{n}$: the rays of a top-dimensional cone $\sigma$ span the lattice $\Lambda_\sigma$ of this cone, and the corresponding boundary divisors of $\overline{M}_{0,n}$ intersect in a point of multiplicity one.
\end{proof}

\begin{theorem}\label{thm-tropicalizing}
  Let $w$ be heavy/light. Consider the embedding \[
  M_{0,w}\hookrightarrow T_w = \Pr_w( T^{\binom{n}{2}}/T^n)
  \] 
 described in Lemma~\ref{lem-embed}. The closure of $M_{0,w}$ the compactification of $T_w$ defined by the fan $\mk{w}$ is isomorphic to $\overline{M}_{0,w}$. The tropicalization of $M_{0,w}$ with respect to this embedding is $\pr_w(\mk{n})=\mk{w}$.
\end{theorem}
\begin{proof}
We wish to show that the map $\overline M_{0,w} \to X(M_{0,w}^{\trop})$ is an embedding. According to~\cite[Lemma 2.6 (4), Theorem 2.10]{HKT09}, this occurs when the following two conditions hold. Let $S$ be a stratum, let $M_{S}$ be $\mathcal O^*(S)/k^*$ and $M^S_{M_{0,w}}$ be the sublattice of $\mathcal O^*(M_{0,w})/k^*$ generated by units having zero valuation on $S$.
\begin{enumerate}
\item For each boundary divisor $D$ containing $S$, there is a unit $u\in \mathcal O^*(M_{0,w})$ with valuation $1$ on $D$ and valuation $0$ on other boundary divisors containing $S$.
\item  $S$ is very affine and the restriction map $M^S_{M_{0,w}}\to M_S$ is surjective.
\end{enumerate}

For (1) observe that for each boundary divisor $D$, we can choose an appropriate forgetful morphism to $M_{0,4}$, informally a \textit{cross ratio} map, as is done in~\cite[Section 5]{Tev07}. It is straightforward to check that these functions have valuation $1$ on $D$ and valuation $0$ on any other boundary divisor. 

We see that all the strata $S$ are very affine in this case. Recall that the divisors containing a given stratum $S$ are precisely those rays of $M_{0,w}^{\trop}$ which are contained in the cone $\sigma_S$ corresponding to $S$. It follows immediately now from the discussion of the boundary stratification of $M_{0,w}$ in Remark~\ref{rem: m0w-boundary} that the restriction map is surjective.

\end{proof}

\begin{remark}
 If we drop the condition of having only heavy and light points, many of the statements discussed here are still true. The geometric tropicalization of $\overline{M}_{0,w}$ using the embedding in Lemma~\ref{lem-embed} still equals $\pr_w(\mk{n})$, however we know already that not all cones are mapped injectively in this case. As a result, the underlying abstract cone complex of $\pr_w(\mk{n})$ is not $\mk{w}$. On the algebraic side, this is reflected by the fact that we still have a map from $\overline{M}_{0,w}$ to the toric variety defined by $\pr_w(\mk{n})$, but it does not map all boundary strata injectively.
\end{remark}

\begin{remark}
Allowing edge lengths to become infinite, analogously to~\cite{ACP12} we obtain an extended cone complex $\overline M^{\trop}_{0,w}$. The arguments above also show that for $w$ heavy/light, the extended tropicalization of $\overline M_{0,w}$ inside the toric variety $X(M_{0,w}^{trop})$ can be identified with $\overline M^{\trop}_{0,w}$. 
\end{remark}

\subsection{Extended example: Losev--Manin spaces}\label{sec: LosevManin}

Let $w$ be the weight vector $(1,1,\epsilon,\cdots,\epsilon)$ for $\epsilon$ light. The space $\overline{M}_{0,w}$ is called the  \textit{Losev--Manin} moduli space and parametrizes chains of projective lines with $n$ marked points, where $n = \ell(w)$. These spaces were introduced and studied in \cite{losevmanin} and play a role, for instance, in the theory of relative stable maps, as a target for branch morphisms.

The Losev--Manin moduli spaces are toric varieties themselves, and as a result the situation simplifies considerably in this case. In fact, there is some beautiful combinatorics that arises in this situation. See~\cite{Bat11} for a proof of the following proposition. See also~\cite{Kap93, losevmanin}.

\begin{prop}
Let $X_n$ be the toric variety obtained by blowing up $\mathbb P^{n-3}$ at all torus invariant subvarieties up to codimension $2$ in order of decreasing codimension. The Losev--Manin moduli space $\overline M_{0,w}$ is isomorphic to $X_n$. 
\end{prop}
The associated fan $\Sigma(X_n)$ is the normal fan of the permutahedron. 

\begin{remark}
In the case of the Losev--Manin moduli space, the \textit{modular boundary}, i.e. the complement of $M_{0,n}$, is not normal crossing. However, the locus of non-smooth curves coincides with the toric boundary (the complement of the big torus), which is simple normal crossing. 
\end{remark}

Since the tropical moduli space $\pr_w(\mk{n})=\mk{w}$ is a complete fan, the embedding of $\overline{M}_{0,w}$ into the corresponding toric variety is surjective. The fact that Losev--Manin spaces are toric can thus also be derived from Theorem \ref{thm-tropicalizing}.

Let us discuss some aspects of this fan more closely. We can use Theorem \ref{moduli_cor_mainresult} and \ref{thm-heavyandlight} and study $B'(G(w))=\pr_w(\mk{n})=\mk{w}$. The graph $G(w)$ is a star graph, i.e.\ it consists of $t$ edges meeting in a single vertex. The matroid of this graph is $U_{t,t}$, so we see that 
 $\mk{w} \cong \R^{t-1}$. Furthermore, the subdivision of $\R^{t-1}$ is the nested set subdivision with respect to the 1-connected flats of $G(w)$. However, all flats of $G(w)$ are 1-connected, so the subdivision is actually the chains-of-flats-subdivision of $U_{t,t}$.

 We can also describe the tropical curves we parametrize more concretely, as follows. Any $w$-stable rational curve is a so-called \textit{caterpillar tree} (Figure \ref{geom_fig_losevmanin}): it consists of a single chain of edges with the \textit{heavy} leaves at either end and the remaining leaves distributed at will along the chain of edges. 
 
 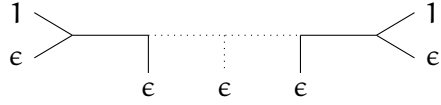
\begin{figure}[ht]
  \centering
  \begin{tikzpicture}
   \draw (0,0) -- (1,0);
   \draw[dotted] (1,0) -- (3,0);
   \draw (3,0) -- (4,0);
   \draw (0,0) -- (-0.5,0.3) node[left]{$1$};
   \draw (0,0) -- (-0.5,-0.3) node[left]{$\epsilon$};
   \draw (1,0) -- (1,-0.5) node[below]{$\epsilon$};
   \draw[dotted] (2,0) -- (2,-0.5) node[below]{$\epsilon$};
   \draw (3,0) -- (3,-0.5) node[below]{$\epsilon$};
   \draw (4,0) -- (4.5,0.3) node[right]{$1$};
   \draw (4,0) -- (4.5,-0.3) node[right]{$\epsilon$};
  \end{tikzpicture}
  \caption{A dual graph of a curve parametrized by a point of Losev--Manin space.}\label{geom_fig_losevmanin}
 \end{figure}

We identify each such curve through its vector of leaf distances 
$$(\dist(l_1,l_3),\dots,\dist(l_1,l_n)) \in \R^t.$$
In turn, each element of $\R^t$ can be considered such a distance vector, if we set its smallest entry to 0. So again, we obtain as parameter space $\R^t / (1,\dots,1) \cong \R^{t-1}$. A canonical subdivision is dictated by the combinatorial types, more precisely, we obtain a top-dimensional cone for each of the $(n-2)!$ orderings on the leaves $l_3,\dots,l_n$. One can easily check that this is the same as the chains-of-flats subdivision of $U_{t,t}$.

\subsection{Spaces of weighted stable curves and Berkovich skeletons}\label{sec: skeleta}
We continue to work over trivially valued $\mathbb C$. Let $X$ be a proper normal variety. Let $U\hookrightarrow X$ be given by the open complement of a normal crossing divisor $D$. The pair $(X,D)$ carries the structure of a toroidal embedding, in the sense of~\cite{KKMSD}. Associated to any toroidal embedding is an extended cone complex $\overline\Sigma(X)$. Thuillier~\cite{Thu07} realizes this cone complex as a skeleton of the Berkovich analytic space $X^{an}$. Builiding on this, Abramovich, Caporaso, and Payne~\cite{ACP12} identify the tropical moduli spaces $\overline M_{g,n}^{\trop}$ with the skeleton of the Berkovich analytification of $\overline{\mathcal{M}}_{g,n}$. They use this formalism to study a functorial tropicalization for this moduli space.

We now extend this to the present situation. To state the results most cleanly, it is convenient to work with the extended cone complex $\overline M^{\trop}_{0,w}$, obtained by allowing edge lengths to become infinite, identical to the $\overline{\mathcal{M}}^{\trop}_{g,n}$ case. 

We define a ``set theoretic'' tropicalization map 
\[
\trop: \overline M_{0,w}^{an}\to \overline M_{0,w}^{\trop},
\]
as follows. Let $p\in  \overline M^{an}_{0,w}$. Such a point $p$ can be represented by a stable curve $[C]$ over a valued extension $K$ of $\mathbb C$. Since $\overline M_{0,w}$ is proper, this extends to curve over the valuation ring $R$ of $K$. Define $\trop(p)$ to be the dual graph $\Gamma_C$ of the special fibre of $[C]$. The edges of $\Gamma_C$ correspond to nodes in the special fibre. Such a node has a defining equation 
\[
xy = f,
\]
where $f\in R$. We assign the corresponding edge length equal to $val(f)$. Note that if nodes appear in the generic fibre, then the defining equation is locally $xy = 0$, and the corresponding edge has length $\infty$.

Let $w$ be heavy/light as before. Recall from the previous section that the complement of the locus of smooth curves in $\overline M_{0,w}$ is a divisor with simple normal crossings. We have the following result.

\begin{theorem}
The (extended) cone complex $\overline M_{0,w}^{\trop}$ is identified with the skeleton of $\overline M_{0,w}^{an}$. Furthermore, there exists a section of the tropicalization map $\trop: \overline M^{an}_{0,w}\to  \overline M^{\trop}_{0,w}$,
\[
s: \overline M^{\trop}_{0,w}\to  \overline M^{an}_{0,w},
\]
which realizes the tropicalization as a skeleton of the Berkovich space. Furthermore, there is a canonical strong deformation retract from $\overline M^{an}_{0,w}$\ onto  $\overline M^{\trop}_{0,w}$. 
\end{theorem}

\begin{proof}
The proof is essentially the same as the corresponding statement for $\overline M_{g,n}$, so we merely provide a sketch. Consider a $0$-stratum of $\overline M_{0,w}$, with respect to the previously described toroidal structure. Let $[C]$ be the $w$-stable curve parametrized by this stratum. The deformation parameters of the nodes of $[C]$ form a system of local coordinates near $[C]\in \overline M_{0,w}$. This furnishes a formal neighborhood of $[C]$ isomorphic to a formal $\mathbb A^N$, where $N$ is the number of nodes of $[C]$. The valuations of these deformation parameters yield coordinates on the top dimensional cone of $M_{0,w}^{\trop}$ corresponding to this zero stratum. However, these are naturally identified with coordinates on the top dimensional cone of $M_{0,w}^{\trop}$ parametrizing tropical $w$-stable curves with underlying combinatorial type given by the dual graph $\Gamma_C$ of $[C]$. For higher dimensional strata, the deformation parameters form a subset of the coordinates of a formal local affine space, which map to lower dimensional cones in the skeleton. The fact that the set theoretic tropicalization map agrees with Thuillier's ``projection to the skeleton'' map is standard and follows from analogous arguments in~\cite[Section~6]{ACP12}.
\end{proof}

Suppose that $w$ is not necessary heavy/light. Let $D$ be the boundary divisor of $\overline M_{0,w}$ given by the union of divisors corresponding to the rays of $\textnormal{pr}_w(M_{0,n}^{\trop})$. The analysis carried out in Section~\ref{section_moduli} characterizes the boundary intersections of components of $D$. More precisely, two irreducible boundary divisors $D_i$ and $D_j$ intersect precisely when the corresponding rays of $\textnormal{pr}_w(M^{\trop}_{0,n})$ span a $2$-dimensional cone $\sigma$. The situation generalizes in the natural way for manifold intersections. Furthermore, the combinatorial type of graphs parametrized by $\sigma$ are dual to the universal curve over $D_i\cap D_j$. Consequently, the boundary divisor $D$ is simple normal crossing, and identical arguments as above yield the following. 

\begin{theorem}
The cone complex $\textnormal{pr}_w(\overline M_{0,n})$ is identified with the Thuillier skeleton of $\overline M_{0,w}^{an}$, with the toroidal structure coming from the inclusion of the complement of the divisor $D$ above. 
\end{theorem}

We return now to the $w$ heavy/light case. The Hassett spaces admit natural tautological morphisms, known as reduction maps. Given two weight data $w = (w_j)$ and $w' = (w'_j)$ such that $w_i\geq w'_i$ for all $i$, there exists a natural birational morphism
\[
\rho_{w,w'}: \overline M_{0,w}\to \overline M_{0,w'},
\]
obtained by collapsing components of curves that become unstable under the weights $w'$. In particular, there always exists a reduction map
\[
\rho_w:\overline M_{0,n}\to \overline M_{0,w}.
\]

\begin{theorem}
The map $\rho_w$ is compatible with the tropical projection maps $pr_w$. More precisely, in the notation above, $\textnormal{pr}_w = \trop\circ \rho_w$. 
\end{theorem}

\begin{proof}
Recall that the formation of skeletons is functorial for toroidal morphisms. Since $\rho_w$ is birational, it suffices to check that for any point $[C]\in \overline M_{0,n}$, there exist formal local toric charts around $x$ and $\rho_w(x)$, such that the monomial coordinates on the target pullback to monomial coordinates on the source. We take the charts to be the ones given by the deformation parameters of $[C]$. Note that the inverse image of a node of $\rho_w([C])$ is a single node of $[C]$. If $\zeta$ is the deformation parameter at the node of $\rho_w([C])$, notice that $\rho_w^*\zeta$ is simply $\widetilde \zeta$ where $\widetilde \zeta$ is the deformation parameter of the corresponding node of $[C]$. The morphism is clearly toric and dominant in the local charts, and the result follows.  
\end{proof}

\section{Spaces of rational weighted stable curves as fibre products}\label{section_fibre}
In this section, we express the projections $\pr_w(\mk{n})$ in terms of fibre products. We use the equality $\pr_w(\mk{n})=B'(G(w))$ from Theorem \ref{moduli_cor_mainresult} and study general properties of fibre products of Bergman fans. If $w$ has only heavy and light points, the tropical description as fibre products matches the analogous algebraic description nicely.

\subsubsection{Hassett spaces with weight $0$ points}\label{sec: alg-fibre-products}

Hassett considers ``zero weight'' variations on the moduli problem for weighted stable curves. Such spaces are natural from the perspective of the log minimal model program. That is, we consider $w = (w_1,\ldots, w_n)$ with $0\leq w_j\leq 1$, and $\sum w_i>2$.

The resulting moduli space $\overline{M}_{0,w}$ can be described as follows. Let $w^+$ be the vector of weights containing the positive entries of $w$, and assume that there are  $t$ entries equal to $0$ in $w$. Following Hassett~\cite[Section 2]{Has03},  the moduli space $\overline M_{0,w}$ is identified with the $t$-fold fibre product of the universal curve $\mathcal C_{0,w^+}$ of $\overline M_{0,w^+}$ over $\overline M_{0,w^+}$,
\[
\overline M_{0,w} \cong \mathcal C_{0,w^+}\times_{\overline M_{0,w^+}}\cdots \times_{\overline M_{0,w^+}}\mathcal C_{0,w^+}
\]

The special case when all positive weights are equal to $1$ allows the universal family to also be identified with the moduli space of curves with one more point of arbitrary weight. In particular, we see that

\[
\overline M(1^f,0^t)\cong\overline{M}_{0,f+1}\times_{\overline{M}_{0,f}}\ldots \times_{\overline{M}_{0,f}}\overline{M}_{0,f+1}.
\] 


Replacing the $0$ weights with $\epsilon$ weights, we obtain a birational morphism
\[
\overline M(1^f, \epsilon^t)\to \overline M(1^f,0^t),
\]
which is a desingularization of $\overline M(1^f,0^t)$. The exceptional loci are described explicitly in~\cite[Corollary 3.5]{Has03}. Informally, $\overline M(1^f,0^t)$ contains a locus parametrizing curves in which multiple $0$-weight points can collide with nodes. We now investigate the tropical analogue

\subsubsection{Tropical fibre products}

\begin{defn}
  Let $f: X \to Y := B(M)$ be a morphism from a tropical fan to a Bergman fan. Assume there are rational functions $\varphi_1,\dots,\varphi_r$ on $Y$ and $C := \varphi_1 \cdot \dots \cdot \varphi_r \cdot Y$ (for an in-depth discussion of rational functions and divisors, see for example \cite{AR07}). Then we define the \textit{pull-back} of $C$ along $Y$ to be 
  $$ f^*C := (\varphi_1 \circ f) \cdot \dots (\varphi_r \circ f) \cdot X.$$
 \end{defn}

 \begin{remark}
  One could of course make this definition for an arbitrary target variety $Y$. However, in this case the pull-back may depend on the choice of rational functions $\varphi_i$. That this is not the case for Bergman fans was shown in \cite[Example 8.2]{FR10}.
  
  We will need this definition in the case, where $Y = B(N) \times B(N) = B(N \oplus N)$ for some matroid $N$ and $C = \Delta_Y := \{(x,x); x \in Y\}$ is the diagonal of $Y$. It was shown in \cite[Corollary 4.2]{FR10} that there are rational function $\varphi_1,\dots,\varphi_r, r = \textnormal{rank}(N)$ on $B(N) \times B(N)$ such that $\Delta_{B(N)} = \varphi_1 \cdot \dots \varphi_r \cdot (B(N) \times B(N))$.
 \end{remark}

\begin{defn}
 Let $f:B(M) \to B(N), g:B(M') \to B(N)$ be morphisms of Bergman fans. Then we define their \textit{fibre product}
 $$B(M) \times_{B(N)} B(M') := (f,g)^*(\Delta_{B(N)}).$$
 Now assume we have morphisms $f': B'(M) \to B'(N), g': B'(M') \to B'(N)$. Both induce morphisms $f: B(M) \to B(N), g: B(M') \to B(N)$ and the corresponding fibre product contains a lineality space $L$ generated by $(1,\dots,1)$. Hence we can define
 $$B'(M) \times_{B'(N)} B'(M') := \left(B(M) \times_{B(N)} B(M') \right) / L.$$ 
\end{defn}

\begin{remark}  
  Tropical fibre products were first defined in \cite{FH11} in the more general context of \textit{smooth} tropical varieties. However, as the definition is a bit more involved and requires notions from intersection theory, we will restrict ourselves to fibre products of Bergman fans.
 
 All examples of fibre products that we consider in this paper will be nice in the sense that they are themselves Bergman fans and their support is equal to the set-theoretic fibre product $\{(x,y) \in B(M) \times B(M'); f(x) = g(y)\}$. However, both statements are false in general: The fibre product need not be a Bergman fan (in fact, it need not even be isomorphic to one!). Also, it is in general strictly contained in the set-theoretic fibre product. In fact, the latter may very well be a cone complex that is not pure or has the wrong dimension.
 
 The tropical fibre product, however, always has the \textit{correct} dimension due to its intersection-theoretic definition: 
 $$\dim \left(B(M) \times_{B(N)} B(M')\right) = \textnormal{rank}(M) + \textnormal{rank}(M') - \textnormal{rank}(N).$$
 This follows from the fact that we apply $\textnormal{rank}(N)$ many rational functions to the tropical fan $B(M) \times B(M')$.
\end{remark}

\subsection{Spaces of rational weighted stable tropical curves as fibre products}

\begin{defn}
 Let $G_0, G_1,G_2$ be graphs and assume $G_1 $ and $G_2$ both contain a subgraph $G_0$. We then denote by $G_1 \times_{G_0} G_2$ the graph obtained by gluing $G_1,G_2$ along these subgraphs.
\end{defn}

\begin{prop}\label{geom_prop_fibreproduct}
 Let $G_1,G_2$ be connected graphs, both containing a subgraph isomorphic to some complete graph $G_0$. Then 
 $$B'(G_1) \times_{B'(G_0)} B'(G_2) \cong B'(G_1 \times_{G_0} G_2).$$
 Furthermore, the support of the left hand side is the set-theoretic fibre product
 $$S(G_1 \times_{G_0} G_2) := \{ (v^1,v^2); v^i \in B'(G_i) \textnormal{ and } v_e^1 = v_e^2 \textnormal{ for all } e \in G_0\}.$$
 \begin{proof}
  First of all note that, since all graphs are connected, we have
  \begin{align*}
    \textnormal{rank}(G_1 \times_{G_0} G_2) &= \abs{V(G_1 \times_{G_0} G_2)} -1\\
	    &= \abs{V(G_1)} + \abs{V(G_2)} - \abs{V(G_0)} - 1\\
	    &= \textnormal{rank}(G_1) + \textnormal{rank}(G_2) - \textnormal{rank}(G_0)
  \end{align*}
In particular, both spaces have the same dimension. We must show that the linear map
\begin{align*}
 i: S(G_1 \times_{G_0} G_2) &\to B'(G_1 \times_{G_0} G_2)  \\
 ( (v_e)_{e \in G_1}, (v_e)_{e \in G_2}) &\mapsto (v_e)_{e \in G_1 \times_{G_0} G_2}
\end{align*}
is an embedding. More precisely, we only have to show that it is well-defined, i.e\ that its image lies in $B'(G_1 \times_{G_0} G_2)$. 

So assume $v^1 \in B'(G_1), v^2 \in B'(G_2)$ and for all edges $e$ in $G_0$ we have $v_e^1 = v_e^2$. Let $v := i(v^1,v^2)$. We want to show that for any circuit $C$ of $G_1 \times_{G_0} G_2$, the maximum of $\{v_e, e \in C\}$ is attained at least twice. If $C$ lies in $G_i$, this is clear, as $v^i \in B'(G_i)$, so it already attains the maximum twice.

Assume $C$ does not lie in $G_i$. We will prove the claim by induction on the cycle length of $C$. If $C$ has length 3, then $C$ must already lie in one of the $G_i$: We assume without loss of generality that two of the edges of $C$ lie in $G_1$. If the last edge is in $G_2$, its vertices must be vertices of $G_0$. But as $G_0$ is complete and a subgraph of $G_1$, the last edge must also lie in $G_1$. If $\abs{C} > 3$, we can use a similar argument to find a \emph{chord} $c$ of $C$, i.e.\ an edge $c \notin C$ connecting two vertices of $C$. This chord subdivides $C$ into two cycles $C',C''$ of smaller length. By induction, the maxima $\max\{v_e; e\in C'\}, \max\{v_e;e \in C''\}$ are assumed twice. If both maxima are assumed away from $c$, the maximum over $C$ is also attained twice. If one of the cycle attains its maximum on $c$, then either the other cycle attains its maximum away from $c$ and it is bigger, so that the maximum over $C$ is again attained twice, or the other cycle also attains its maximum on $c$, and they are equal. 
In that case the maximum is attained on $v_c,v_{e'},v_{e''}$ for edges $e' \in C'\wo \{c\}, e'' \in C''\wo\{c\}$. In either case, the maximum is attained twice on $C$.

As $B'(G_1 \times_{G_0} G_2)$ is an irreducible tropical variety containing $i(S(G_1 \times_{G_0} G_2))$ and has the same dimension as $B'(G_1) \times_{B'(G_0)} B'(G_2)$, we see that we must have
$$i(\abs{B'(G_1) \times_{B'(G_0)} B'(G_2)}) =i( S(G_1 \times_{G_0} G_2)) = \abs{B'(G_1 \times_{G_0} G_2)},$$
which implies that $\abs{B'(G_1) \times_{B'(G_0)} B'(G_2)} =  S(G_1 \times_{G_0} G_2)$.

In particular $B'(G_1) \times_{B'(G_0)} B'(G_2)$ is isomorphic to a multiple of $B'(G_1 \times_{G_0} G_2)$. It remains to show that the weights of the fibre product (or equivalently: at least one weight) are 1. But this follows from Lemma \ref{intro_lemma_fibreweight}.
  \end{proof}
\end{prop}
\begin{lemma}\label{intro_lemma_fibreweight}
 Let $M_1,M_2$ be matroids on ground sets $E_1,E_2$, where $E_0 := E_1 \cap E_2 \neq \emptyset$. Assume $M_0 = M_{1 \mid E_0} = M_{2 \mid E_0}$ is the common restriction of both matroids. Let $p_i: B(M_i) \to B(M_0)$ be the corresponding projections of Bergman fans. If the support of the fibre product is the set-theoretic fibre product
 $$\abs{B(M_1) \times_{B(M_0)} B(M_2)} = \{ (a_1,a_2) \in B(M_1) \times B(M_2); p_1(a_1) = p_2(a_2)\},$$
 then the fibre product has weight 1 on each top-dimensional cone.
 \begin{proof}
 Assume $B(M_1),B(M_2), B(M_0)$ are all equipped with the polyhedral structure defined by their chains of flats. In particular $p_i$ maps cones of $B(M_i)$ to cones of $B(M_0)$. Then the set-theoretic fibre product has a natural polyhedral structure through      the conewise fibre product:
  $$\curly{P} := \{ \sigma_1 \times_\tau \sigma_2; \sigma_i \textnormal{ a cone of } B(M_i) \textnormal{ and } p_i(\sigma_i) = \tau\},$$
  where $\sigma_1 \times_\tau \sigma_2 = \{(a_1,a_2); a_i \in \sigma_i; p_1(a_1) = p_2(a_2)$.
  
  In particular, a top-dimensional cone $\sigma$ of $\curly{P}$ is of the form $\sigma_1 \times_\tau \sigma_2$, where $\sigma_1,\sigma_2,\tau$ are all top-dimensional cones of their respective fans and an interior point $q$ of $\sigma$ is mapped to an interior point of $\tau$. To compute the weight of $\sigma$, we can look at the local morphism 
  $$\Star_{B(M_1) \times B(M_2)}(q) \to \Star_{B(M_0) \times B(M_0)}( (p_1,p_2)(q))$$
  induced by the combined projections $(p_1,p_2): B(M_1) \times B(M_2) \to B(M_0) \times B(M_0)$. Since $q$ is an interior point of $\sigma$ and $(p_1,p_2)(q)$ is an interior point of $\tau$, this is just a projection map of linear spaces
  $$\R^{r_1 + r_2} \to \R^{2r_0},$$
  where $r_i = \textnormal{rank}(M_i)$. Applying the diagonal functions $\varphi$ defined in \cite{FR10} cuts out the diagonal $\Delta_{\R^{r_0}}$ with weight 1 on the right hand side. Pulling this back via a linear projection yields again a linear space with weight 1. This concludes the proof.
 \end{proof}
\end{lemma}

\begin{defn}
 Let $w$ be a weight vector. We divide the vertices of $G(w)$ into
\begin{align*}
 L(w) &:= \{i \in \{2,\dots,n\}, w_i > 1/2\} \\
 S(w) &:= \{i \in \{2,\dots,n\},w_i \leq 1/2\}
\end{align*}
 We define the following graphs:
 \begin{itemize}
  \item $K(w)$ is the restriction of $G(w)$ to $L(w)$ (i.e.\ it is the complete graph on $L(w)$).
  \item For $i \in S(w)$, we let $G_i(w)$ be the restriction of $G(w)$ to $L(w) \cup \{i\}$.
 \end{itemize}
\end{defn}

\begin{remark}
 The graphs $G_i(w)$ share the common subgraph $K(w)$, and gluing them together gives us back our weight graph:
 $$G(w) = G_{i_1}(w) \times_{K(w)} \dots \times_{K(w)} G_{i_s}(w),$$
 where $S(w) = \{i_1,\dots,i_s\}$. As $K(w)$ is complete, we can apply Proposition \ref{geom_prop_fibreproduct}. However, to obtain a deeper geometric understanding, we first want to study the spaces $B'(G_i(w))$.
\end{remark}

\begin{prop}\label{geom_prop_single_degree}
 Assume $i \in S(w)$ has vertex degree 1. Then 
 $$B'(G_i(w)) \cong \mk{\abs{L(w)}+1} \times \R$$
 \begin{proof}
  By assumption $G_i(w)$ is just $K(w)$ with an additional edge attached. Hence its matroid is just the direct sum of the matroid of $K(w)$ and a coloop. This implies the claim.
 \end{proof}
\end{prop}

\begin{remark}
 This has a natural geometric interpretation that is very similar to Example \ref{moduli_ex_badtypes}. The fact that vertex $i$ is only connected to one vertex $j$ of $K(w)$ means that there are exactly two leaves, namely 1 and $j$, with which $i$ is compatible. More precisely, $w_i + w_j, w_i + w_1 > 1$, but $w_i + w_k \leq 1$ for all other $k$. Hence we can only place leaf $i$ along the \textit{line} spanned by leaves $j$ and $1$ (All other choices would produce inherited unstable types).
\end{remark}

\begin{prop}\label{geom_prop_higher_degree}
 Assume $i \in S(w)$ has vertex degree $d > 1$. Then
 $$B'(G_i(w)) \cong \mk{\abs{L(w)}+1} \times_{\mk{d+1}} \mk{d+2}.$$
 \begin{proof}
  We can write $G_i(w)$ as a glued graph: Let $K_i$ be the restriction of $G_i(w)$ to vertices $j \neq i$ connected to $i$. In particular it is a complete graph on $d$ vertices. If we add vertex $i$, we obtain $K_i'$, the complete graph on $d+1$ vertices. Now $K(w)$ and $K_i'$ share the common subgraph $K_i$ and 
  $$G_i(w) = K(w) \times_{K_i} K_i'.$$
  The result now follows from Proposition~\ref{geom_prop_fibreproduct}.
 \end{proof}
\end{prop}


\begin{theorem}\label{thm-tropfibre}
 Let $w $ be a weight vector. Let $D_1 := \{i \in S(w): \deg(i) = 1\}$ and assume $S(w) \wo D_1 = \{i_1,\dots i_s\}$. Let
 \[
   \curly{M} := \prod_{j=1}^s \mk{\deg(i_j)+1} \ \ \   \curly{M}' :=  \prod_{j=1}^s \mk{\deg(i_j)+2}
 \]
%
Then we have
 $$\pr_w(\mk{n}) \cong \R^{\abs{D_1}} \times \left( \mk{\abs{L}+1} \times_{\curly{M}} \curly{M}'\right),$$
 where the maps from $\mk{\abs{L}+1}$ and $\curly{M}'$ to $\curly{M}$ are the natural tuples of forgetful maps (If $s = 0$, we set $\mk{\abs{L}+1} \times_{\curly{M}} \curly{M} := \mk{\abs{L}+1}$).
 \begin{proof}
  Assume $D_1 = \{k_1,\dots,k_t\}$. By Proposition \ref{geom_prop_fibreproduct} we know that
  $$\pr_w(\mk{n}) \cong B'(G_{k_1}(w)) \times_{B'(K(w))} \dots \times_{B'(K(w))} B'(G_{i_s}(w))$$
  and that the support of the latter is the set-theoretic fibre product. The claim now follows from Propositions \ref{geom_prop_single_degree} and \ref{geom_prop_higher_degree}.
 \end{proof}
\end{theorem}


\begin{corollary}\label{cor-fibreheavylight}
 Let $w = (1^f,\epsilon^t)$, where $f \geq 3, t \geq 2$. Then 
 \[
 \mk{w} \cong \underbrace{\mk{f+1} \times_{\mk{f}} \dots \times_{\mk{f}} \mk{f+1}}_{t \textnormal{ times}}.
\]
\end{corollary}

\begin{remark}
The last result might seem somewhat incongruous at first, as classically, we have
 \[
\overline M(1^f,0^t)\cong\overline{M}_{0,f+1}\times_{\overline{M}_{0,f}}\ldots \times_{\overline{M}_{0,f}}\overline{M}_{0,f+1}
\] 
and that $\overline{M}(1^f,\epsilon^t)$ is a blowup of this. In fact, the tropical phenomena are analogous.

The tropical fibre product has a canonical polyhedral structure $\curly{P}$ consisting of taking conewise fibre products of cones of $\mk{f+1}$ in its coarse subdivision. Points in the same cone correspond to $n$-marked tropical curves that
\begin{itemize}
 \item have the same combinatorial \textit{base curve} $C$ when forgetting all light ends;
 \item have their light ends placed on the same edges or ends of $C$.
\end{itemize}

In other words, this is not the polyhedral structure that was assigned to $\mk{w}$. Whenever multiple light ends are put on the same edge of the base curve, we need to subdivide. More precisely, if $\sigma$ is a cone in $\curly{P}$ corresponding to a base curve with $d$ bounded edges and $t$ light ends placed on its edges and leafs, then we can identify $\relint(\sigma)$ with $\R^{d+t}$. Under this identification, the subdivision is a product of $\R^k$ together with tropical Losev--Manin spaces. Recall from Section~\ref{sec: LosevManin} that $\mk{(1^2,\epsilon^r)} \cong \R^{r-1}$ parametrizes ways to position $r$ labelled ends relative to each other on a line. For each bounded edge with at least one light end, we get a factor of $\R^2 \times \mk{(1^2,\epsilon^r)}$, where $r$ is the number of light ends placed on this edge. In similar fashion, we obtain a factor of $\R \times \mk{(1^2,\epsilon^r)}$ for each end with at least one light end placed on it. Finally, we get a factor of $\R$ for each bounded edge with no light ends. 
\end{remark}

\bibliographystyle{siam}
\bibliography{tropWSC}

\end{document}